\documentclass[english]{article}
\usepackage[T1]{fontenc}
\usepackage[utf8]{inputenc}
\usepackage{geometry}
\usepackage{color}
\usepackage[english]{babel}
\usepackage{float}
\usepackage{mathtools}
\usepackage{bm}
\usepackage{amsmath}
\usepackage{amssymb}
\usepackage{graphicx}
\usepackage[unicode=true,
 bookmarks=false,
 breaklinks=false,pdfborder={0 0 0},pdfborderstyle={},backref=false,colorlinks=true, citecolor=blue]
 {hyperref}

\makeatletter

\floatstyle{ruled}
\newfloat{algorithm}{tbp}{loa}
\providecommand{\algorithmname}{Algorithm}
\floatname{algorithm}{\protect\algorithmname}

\usepackage{babel}

\usepackage{algorithmic}

\usepackage{amsthm}

\theoremstyle{plain}

\newtheorem{lemma}{\textbf{Lemma}}
\newtheorem{theorem}{\textbf{Theorem}}

\newtheorem{assumption}{\textbf{Assumption}}

\theoremstyle{definition}
\newtheorem{remark}{\textbf{Remark}}

\usepackage{color}
\definecolor{cm}{RGB}{0,0,200}

\definecolor{yy}{RGB}{255,0,255}

\newcommand{\est}{\texttt{square-root MC}}

\allowdisplaybreaks

\makeatother

\begin{document}
\title{Optimal tuning-free convex relaxation for noisy matrix completion}
\author{{Yuepeng Yang\thanks{Department of Statistics, University of Chicago; Email: \texttt{\{yuepengyang,
congm\}@uchicago.edu}}} \and{Cong Ma\footnotemark[1]}}
\maketitle
\begin{abstract}
This paper is concerned with noisy matrix completion---the problem
of recovering a low-rank matrix from partial and noisy entries. Under uniform sampling and incoherence assumptions, we
prove that a tuning-free square-root matrix completion estimator (\est) 
achieves optimal statistical performance for solving the noisy matrix
completion problem. Similar to the square-root Lasso estimator in
high-dimensional linear regression, \est~does not rely on the knowledge
of the size of the noise. While solving \est~is a convex program, our statistical
analysis of \est~hinges on its intimate connections to a nonconvex
rank-constrained estimator. 
\end{abstract}

\section{Introduction}

Low-rank matrix completion~\cite{candes2009exact,keshavan2010matrix}
aims to reconstruct a low-rank data matrix from its partially observed
entries. This problem finds numerous applications in collaborative
filtering~\cite{rennie2005fast}, causal inference~\cite{athey2021matrix},
sensor network localization~\cite{biswas2006semidefinite}, etc. 

In this paper, we focus on the \emph{noisy} matrix completion problem, in
which the revealed entries are further corrupted by random noise.
Mathematically, let $\bm{L}^{\star}\in\mathbb{R}^{n\times n}$ be
a rank-$r$ matrix of interest, and $\bm{E}\in\mathbb{R}^{n\times n}$
denotes the noise matrix. We observe a subset of entries 
\begin{equation}
M_{ij}=L_{ij}^{\star}+E_{ij},\qquad\text{for }(i,j)\in\Omega,\label{eq:obs}
\end{equation}
where $\Omega\subseteq\{1,2,\ldots,n\}\times\{1,2,\ldots,n\}$ represents
the index set of the observations. The goal of noisy matrix completion
is to recover the underlying low-rank matrix $\bm{L}^{\star}$ given
the observation $\bm{M}=[M_{ij}]$. 

Arguably, one of the most natural approaches to solving noisy matrix
completion is the following nuclear norm regularized least-squares
estimator~\cite{candes2010matrix,chen2020noisy}: 
\begin{equation}
\min_{\bm{L}\in\mathbb{R}^{n\times n}}\quad\sum_{(i,j)\in\Omega}(L_{ij}-M_{ij})^{2}+\lambda\|\bm{L}\|_{*},\label{eq:mc}
\end{equation}
where $\|\bm{L}\|_{*}$ denotes the nuclear norm (i.e., sum of singular
values) of the matrix $\bm{L}$, and $\lambda>0$ is a tuning parameter.
Here, the least-squares loss $\sum_{(i,j)\in\Omega}(L_{ij}-M_{ij})^{2}$
measures the fidelity of the estimate $\bm{L}$ to the observation
$\bm{M}$, while the nuclear norm penalty $\lambda\|\bm{L}\|_{*}$
encounrages the low-rank property of the solution. In a recent work~\cite{chen2020noisy},
it has been shown that with properly chosen regularization parameter
$\lambda$, the nuclear norm regularized least-squares estimator~\eqref{eq:mc}
achieves optimal statistical performance in terms of estimating the
low-rank matrix $\bm{L}^{\star}$. However, this optimal choice depends on the noise size, which is often \emph{unknown
}in practice. This begs the question:

\medskip

\emph{Can we develop an estimator for noisy matrix completion that
does not rely on the unknown noise size (a.k.a., tuning-free), and
at the same time achieves optimal statistical performance?}

\medskip

Motivated by the success of the square-root Lasso estimator~\cite{belloni2011square}
for sparse recovery problems, we consider in this paper the following
square-root matrix completion estimator (dubbed \est):

\begin{equation}
\min_{\bm{L}\in\mathbb{R}^{n\times n}}\quad\sqrt{\sum_{(i,j)\in\Omega}(L_{ij}-M_{ij})^{2}}+\lambda\|\bm{L}\|_{*}.\label{eq:sqrt-mc}
\end{equation}
A notable difference from the vanilla least-squares estimator~\eqref{eq:mc}
is that \est~\eqref{eq:sqrt-mc} aims at minimizing the regularized
$\ell_{2}$ error instead of the regularized \emph{squared} $\ell_{2}$ error. 

\paragraph{Our contributions. }The main result of this paper (cf.~Theorem~\ref{thm:main}) shows
that \est~\eqref{eq:sqrt-mc} with a noise-size-oblivious choice
$\lambda\asymp1/\sqrt{n}$ (e.g., $\lambda=32/\sqrt{n}$) achieves
the optimal error guarantees for recovering the low-rank matrix $\bm{L}^{\star}$
over a wide range of noise sizes.
Such guarantees are on par with those established for the vanilla least-squares
estimator~\eqref{eq:mc} with a choice of $\lambda$ depending on the
noise size~\cite{chen2020noisy}. Clearly, the tuning-free property and statistical optimality
of \est~together answer our motivating question in the
affirmative. 

To put our contributions into context, we would like to immediately
point out two relevant pieces of prior work, while deferring other
related ones to Section~\ref{sec:Prior-art}. First and foremost,
a variant of the \est~estimator has been proposed and studied by
Klopp \cite{klopp2014noisy}, in which an extra element-wise max norm constraint
is added to the problem~\eqref{eq:sqrt-mc}. In the same paper, it
was shown that \est~achieves optimal statistical performance when
the size of the noise is sufficiently large compared to the entries
of the low-rank matrix. However, when the noise size is relatively
small, the upper bound proved therein fails to uncover the optimal
performance of the \est~estimator. In particular, it falls short
of uncovering the exact recovery property when there is no noise,
i.e., when $\bm{E}=\bm{0}$. More recently, Zhang et~al.~\cite{zhang2021square}
focuses on a closely related noisy robust PCA problem~\cite{candes2011robust,chen2021bridging}
and studies a similar tuning-free estimator. Their results, however,
even in the full observation setting (i.e., $\Omega=\{1,2,\ldots,n\}\times\{1,2,\ldots,n\}$),
has a poor dependence on the problem dimension $n$, which is far
from optimality. Detailed comparisons between our results and those
in the papers \cite{klopp2014noisy,zhang2021square} can be found
in Section~\ref{sec:Main-results}.

In establishing the optimal performance of \est, we make the following technical contributions. First, we introduce a new decision variable $\theta$ to convert a non-smooth loss function to a smooth one to facilitate later analysis. We then establish a novel connection between the convex \est~estimator and a smooth nonconvex estimator. In the end, we manage to show that an iterative algorithm allows one to find a statistically optimal solution to the nonconvex program. While this general proof strategy has been laid out in~\cite{chen2020noisy}, novel considerations need to be taken to handle the non-smooth loss function and the new decision variable $\theta$. We defer detailed discussions to relevant places in later analysis. 

\paragraph{Notation.} For a vector $\bm{v}$, we use $\|\bm{v}\|_{2}$
to denote its Euclidean norm. For a matrix $\bm{M}$, we use $\|\bm{M}\|$,$\|\bm{M}\|_{\mathrm{F}}$, and $\|\bm{M}\|_{\infty}$
to denote its spectral norm, Frobenius norm, and the elementwise $\ell_{\infty}$
norm. In addition, $\|\bm{M}\|_{2,\infty}$ denotes the largest $\ell_{2}$
norm of the rows. We also use $\sigma_{j}(\bm{M})$ to denote the
$j$-th largest singular value of $\bm{M}$.

Additionally, the standard notation $f(n)=O\left(g(n)\right)$ or $f(n)\lesssim g(n)$ means that there exists a constant $c>0$ such that $\left|f(n)\right|\leq c|g(n)|$,  $f(n)\gtrsim g(n)$ means that there exists a constant $c>0$ such that $|f(n)|\geq c\left|g(n)\right|$. Also, $f(n)\gg g(n)$ means that there exists some large enough constant $c>0$ such that $|f(n)|\geq c\left|g(n)\right|$. Similarly, $f(n)\ll g(n)$ means that there exists some sufficiently small constant $c>0$ such that $|f(n)|\leq c\left|g(n)\right|$.  %$f(n)=\Theta\left(g(n)\right)$ and $f(n)\asymp g(n)$ means that there exist constants $c_{1},c_{2}>0$ such that $c_{1}|g(n)|\leq|f(n)|\leq c_{2}|g(n)|$.  Also, $f(n)\gg g(n)$ means that there exists some large enough constant $c>0$ such that $|f(n)|\geq c\left|g(n)\right|$. Similarly, $f(n)\ll g(n)$ means that there exists some sufficiently small constant $c>0$ such that $|f(n)|\leq c\left|g(n)\right|$.

\section{Main results \label{sec:Main-results}}

We start with introducing the model assumptions for noisy matrix completion.
The first assumption is on the observation pattern. 

\begin{assumption}\label{assumption:p} Each index $(i,j)$ belongs
to the set $\Omega$ independently with probability $p$. \end{assumption}

The next assumption is concerned with the noise matrix.

\begin{assumption}\label{assumption:random-noise} 

The noise matrix $\bm{E}=[E_{ij}]$ is composed of i.i.d.~zero-mean
sub-Gaussian random variables with variance $\sigma^{2}$ and sub-Gaussian
norm $O(\sigma)$, i.e., $\|E_{i,j}\|_{\psi_{2}}=O(\sigma)$; see Definition 5.7 in the article~\cite{vershynin2010introduction}. 

\end{assumption}

% \begin{remark}Without loss of generality, we will assume $\|E_{i,j}\|_{\psi_{2}}=\sigma$
% for the remaining of this paper. This induces at most an extra constant
% factor in the error bounds.\end{remark} \cm{Why do we need this remark?}

In the end, we turn to the assumptions on the groundtruth matrix $\bm{L}^{\star}$.
Let $\sigma_{\min},\sigma_{\max}$ be the smallest and largest singular
values of $\bm{L}^{\star}$, respectively, and let $\kappa\coloneqq\sigma_{\max}/\sigma_{\min}$
be its condition number. We require the matrix $\bm{L}^{\star}$ to
be $\mu$-incoherent defined in the following way. 

\begin{assumption}\label{assumption:incoherence} The rank-$r$ matrix
$\bm{L}^{\star}$ with SVD $\bm{L}^{\star} = \bm{U}^{\star}\bm{\Sigma}^{\star}\bm{V}^{\star\top}$
is $\mu$-incoherent in the sense that 
\[
\|\bm{U}^{\star}\|_{2,\infty}\le\sqrt{\frac{\mu}{n}}\|\bm{U}^{\star}\|_{\mathrm{F}}=\sqrt{\frac{\mu r}{n}},\qquad\text{and}\qquad\|\bm{V}^{\star}\|_{2,\infty}\le\sqrt{\frac{\mu}{n}}\|\bm{V}^{\star}\|_{\mathrm{F}}=\sqrt{\frac{\mu r}{n}}.
\]

\end{assumption}

Now we are in position to state our main results regarding the \est~estimator, with the proof deferred to Section~\ref{sec:main-proof}.  

\begin{theorem}\label{thm:main} Suppose that Assumptions \ref{assumption:p}-\ref{assumption:incoherence}
hold. In addition, assume that the sample size and the noise level
satisfy 
\[
n^{2}p\ge C_{\mathrm{sample}}\kappa^{4}\mu^{2}r^{2}n\log^{3}n,\qquad\text{and}\qquad\frac{\sigma}{\sigma_{\min}}\sqrt{\frac{n}{p}}\le \frac{C_{\mathrm{noise}} }{ \sqrt{\kappa^{4}\mu r\log n} }
\]
for some sufficient large (resp.~small) constant $C_{\mathrm{sample}}>0$
(resp.~$C_{\mathrm{noise}}>0$). Set $\lambda=C_{\lambda}/\sqrt{n}$
for the \est~estimator~\eqref{eq:sqrt-mc}, where $C_{\lambda}$
is some large absoulute constant (e.g., 32). With probability at least $1-O(n^{-3})$,
any solution $\bm{L}_{\mathrm{cvx}}$ to the \est~problem \eqref{eq:sqrt-mc}
obeys 

\begin{subequations}
\begin{align}
\|\bm{L}_{\mathrm{cvx}}-\bm{L}^{\star}\|_{\mathrm{F}} & \le C_{\mathrm{F}}\kappa\frac{\sigma}{\sigma_{\min}}\sqrt{\frac{n}{p}}\|\bm{L}^{\star}\|_{\mathrm{F}};\label{eq:L-bound-fro}\\
\|\bm{L}_{\mathrm{cvx}}-\bm{L}^{\star}\|_{\infty} & \le C_{\mathrm{\infty}}\sqrt{\kappa^{3}\mu r}\frac{\sigma}{\sigma_{\min}}\sqrt{\frac{n\log n}{p}}\|\bm{L}^{\star}\|_{\mathrm{\infty}};\label{eq:L-bound-infty}\\
\|\bm{L}_{\mathrm{cvx}}-\bm{L}^{\star}\| & \le C_{\mathrm{op}}\frac{\sigma}{\sigma_{\min}}\sqrt{\frac{n}{p}}\|\bm{L}^{\star}\|.\label{eq:L-bound-spec}
\end{align}
\end{subequations} Here $C_{\mathrm{F}},C_{\infty},C_{\mathrm{op}}>0$
are three universal constants.\end{theorem}

Several remarks on Theorem~\ref{thm:main} are in order. 

\paragraph{Minimax-optimal $\ell_{\mathrm{F}}$ estimation error.}

When the condition number $\kappa$ is of a constant order, the \est~estimator
enjoys minimax-optimal $\ell_{\mathrm{F}}$ estimation error \cite{negahban2012restricted,chen2020noisy}.
In contrast, the upper bound in the paper~\cite{klopp2014noisy}
reads $\|\bm{L}_{\mathrm{cvx}}-\bm{L}^{\star}\|_{\mathrm{F}}\lesssim\max\left\{ \sigma,\|\bm{L}^{\star}\|_{\infty}\right\} \sqrt{n\log n/p}$,
which is only statistically optimal when $\sigma\gtrsim\|\bm{L}^{\star}\|_{\infty}$. 
In addition, translating the bound in the paper \cite{zhang2021square} from robust
PCA to the matrix completion setting, one obtains $\|\bm{L}_{\mathrm{cvx}}-\bm{L}^{\star}\|_{\mathrm{F}}\lesssim\sigma n^{2}$,
which has a much worse (and hence sub-optimal) dependence on the problem dimension $n$. 

\paragraph{Noise, sample complexity, and dependency on $\kappa, r$.}
Our assumption on noise level and sample complexity is consistent with \cite{chen2020noisy}. Furthermore, these assumptions are necessary for a non-trivial guarantee as otherwise, a naive zero estimator would achieve the optimal rate. Regarding $\kappa$ and $r$, while we mostly focus on the case where they are of constant size, their dependency on the error rate can be of interest. In particular, the dependency on $r^2$ is the best rate known and is consistent with some other nonconvex methods \cite{zheng2016convergence, chen2020nonconvex}. However, the exact sharp dependency of both $r$ and $\kappa$ remains an open problem. \cite{chen2020noisy} also discusses these in their marks in Section 1. 

\paragraph{Tuning-free property.}

More importantly, the optimal performance of \est~is
achieved in a completely tuning-free fashion. The regularization parameter
$\lambda$ can be set to be $32/\sqrt{n}$, that does not depend on
the noise variance $\sigma^{2}$, the observation probability $p$,
nor the true rank~$r$ of the matrix $\bm{L}^{\star}$. This is in stark contrast to the vanilla
nuclear norm regularized least-squares estimator~\eqref{eq:mc} in which $\lambda$ is set to be on the order of $\sigma \sqrt{n p}$ (cf.~\cite{chen2020noisy}). 

\paragraph{Entrywise error guarantees.}

Also, our main results provide upper bounds on the entrywise estimation
error (cf.~bound~\eqref{eq:L-bound-infty}). Compared to the $\ell_{\mathrm{F}}$ estimation error~\eqref{eq:L-bound-fro}, it can
be seen that the \est~estimator is uniformly good in the sense that
there is no spiky entry estimate with large estimation error. 

\bigskip 
To corroborate our main results, we perform numerical experiments
on noisy matrix completion with simulated data. We fix the rank $r$ to be 5 
throughout the experiment. For
each problem dimension $n$, we generate two $n\times r$ random orthonormal
matrices as $\bm{X}^{\star}$ and $\bm{Y}^{\star}$ and take $\bm{L}^{\star}\coloneqq\bm{X}^{\star}\bm{Y}^{\star\top}$
as the rank-$r$ $n\times n$ groundtruth matrix. The entrywise noise is taken to be Gaussian with variance $\sigma^{2}$. For all the experiments, we set $\lambda=4/\sqrt{n}$ in \est, and report the average results over 20 Monte-Carlo simulations.  
Figure~\ref{fig:relative-error-n-sigma} reports the relative error
of the \est~estimator in Frobenius, spectral, and infinity norms.
More specifically, Figure~\ref{fig:relative-error-n-sigma}(a) fixes $n=500$, $p=0.5$, and varies $\sigma$; Figure~\ref{fig:relative-error-n-sigma}(b) fixes $\sigma=10^{-4}$, $p=0.5$, and varies $n$; Figure~\ref{fig:relative-error-n-sigma}(c) fixes $\sigma=10^{-4}$, $n=2000$, and varies $p$. Overall, the plots showcase a linear relationship between the
performance and the noise size $\sigma$, the problem dimension
$\sqrt{n}$, and the observation probability $p$. This is consistent with the $O(\sigma\sqrt{n/p})$ scaling
proved in Theorem~\ref{thm:main}. 

\begin{figure}
\begin{minipage}[t]{0.33\columnwidth}%
\begin{center}
\includegraphics[width=1\textwidth]{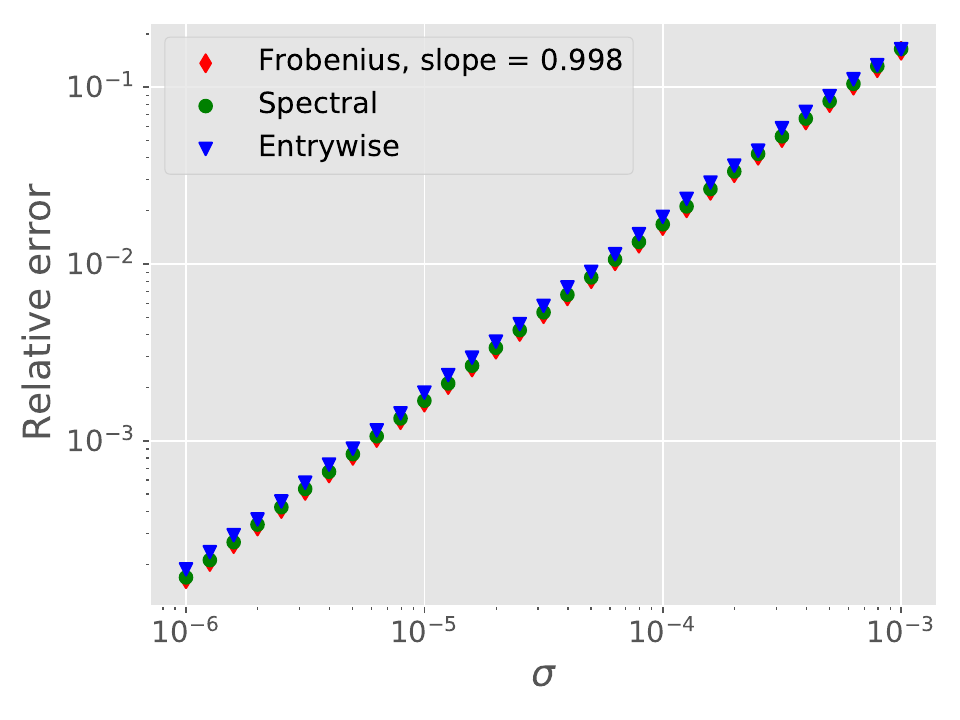}
\par\end{center}
\begin{center}
(a)
\par\end{center}%
\end{minipage}\hfill{}%
\begin{minipage}[t]{0.33\columnwidth}%
\begin{center}
\includegraphics[width=1\textwidth]{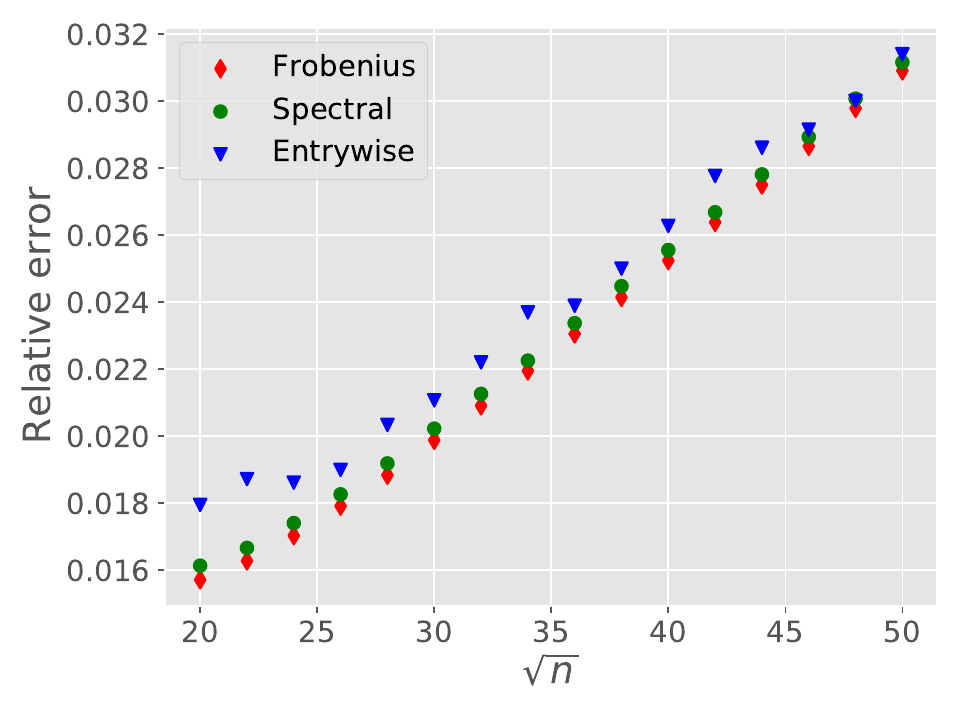}
\par\end{center}
\begin{center}
(b)
\par\end{center}%
\end{minipage}
\hfill{}%
\begin{minipage}[t]{0.33\columnwidth}%
\begin{center}
\includegraphics[width=1\textwidth]{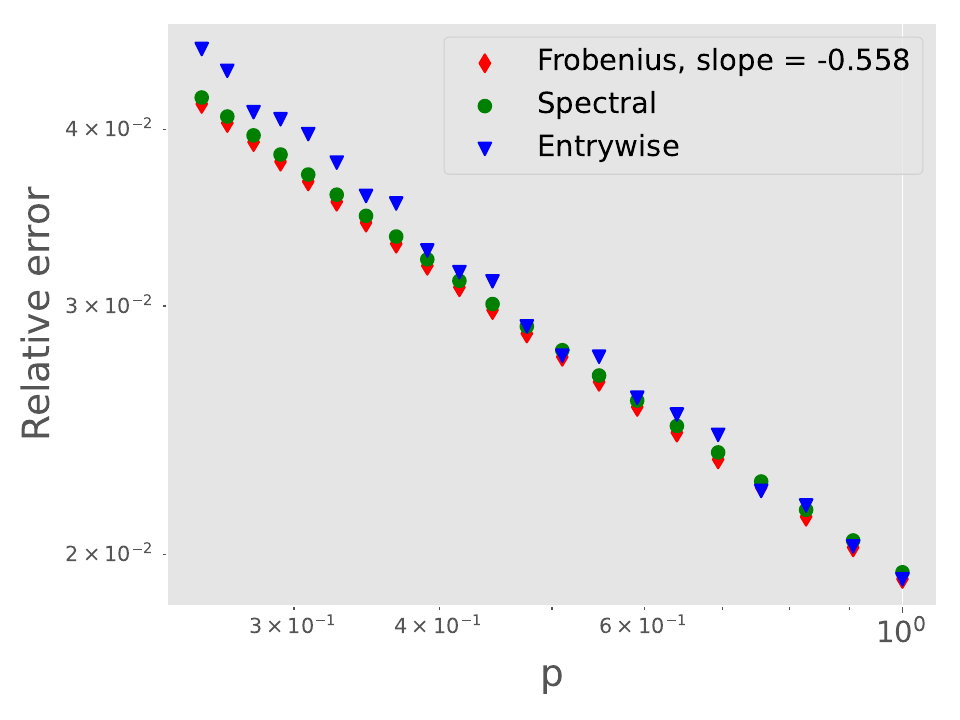}
\par\end{center}
\begin{center}
(c)
\par\end{center}%
\end{minipage}

\caption{\label{fig:relative-error-n-sigma}(a) Relative estimation error of $\bm{L}_{\mathrm{cvx}}$
vs.~noise size $\sigma$ on a log-log scale, where we fix $n=500,r=5,p=0.5$;
(b) Relative estimation error of $\bm{L}_{\mathrm{cvx}}$
vs.~problem size $\sqrt{n}$, where we fix $r=5,\sigma=10^{-4},p=0.5$;
(c) Relative estimation error of $\bm{L}_{\mathrm{cvx}}$
vs.~observation probability $p$ on a log-log scale, where we fix $n=2000, r=5,\sigma=10^{-4}$.
For all three plots, $\lambda=4/\sqrt{n}$ and each point represents the
average of 20 independent trials. }
\end{figure}

\section{Outline of the proof } \label{sec:main-proof}

In this section, we provide the key steps for proving our main result,
i.e., Theorem~\ref{thm:main}. The proof follows the general strategy
of bridging convex and nonconvex solutions, first appeared in the
paper~\cite{chen2020noisy}, with several important modifications
to handle the non-smooth $\ell_{\mathrm{F}}$ norm (as opposed to
the smooth squared $\ell_{\mathrm{F}}$ norm). 

A central object in our analysis is the following nonconvex optimization
problem 

\begin{equation}
\min_{\text{\ensuremath{\bm{X}},\ensuremath{\bm{Y}\in\mathbb{R}^{n\times r}}, \ensuremath{\theta>0} }}\quad f(\boldsymbol{X},\boldsymbol{Y},\theta)\coloneqq\frac{1}{2}\left(\frac{\|\mathcal{P}_{\Omega}(\bm{X}\bm{Y}^{\top}-\bm{M})\|_{\mathrm{F}}^{2}}{\theta}+\theta\right)+\frac{\lambda}{2}\left(\|\bm{X}\|_{\mathrm{F}}^{2}+\|\bm{Y}\|_{\mathrm{F}}^{2}\right),\label{eq:nonconvex-opt-theta}
\end{equation}
which is closely related to the original
convex \est~formulation~\eqref{eq:sqrt-mc}. To see this, first, for any rank-$r$
matrix $\bm{Z}$, one has 
\[
\|\bm{Z}\|_{*}=\inf_{\ensuremath{\bm{X}},\ensuremath{\bm{Y}\in\mathbb{R}^{n\times r}}:\bm{X}\bm{Y}^{\top}=\bm{Z}}\tfrac{1}{2}\left(\|\bm{X}\|_{\mathrm{F}}^{2}+\|\bm{Y}\|_{\mathrm{F}}^{2}\right).
\]
Second and more importantly, we have that for any matrix $\bm{Z}=\bm{X}\bm{Y}^{\top}$,
\[
\|\mathcal{P}_{\Omega}(\bm{Z}-\bm{M})\|_{\mathrm{F}}=\inf_{\theta>0}\quad\frac{1}{2}\left(\frac{\|\mathcal{P}_{\Omega}(\bm{X}\bm{Y}^{\top}-\bm{M})\|_{\mathrm{F}}^{2}}{\theta}+\theta\right).
\]
It turns out that the (approximate) solution to the nonconvex optimization
problem~\eqref{eq:nonconvex-opt-theta} serves as an extremely tight
approximation to the \est~estimator, which facilitates the statistical
analysis of the latter. 

In sum, our proof involves two main steps:
\begin{enumerate}
\item We first show---via an explicit construction---that an approximate
stationary point $\bm{L}_{\mathrm{ncvx}}$ of the nonconvex problem
\eqref{eq:nonconvex-opt-theta} exists and is also close to the groundtruth
matrix $\bm{L}^{\star}$. 
\item We then establish that such an approximate stationary point $\bm{L}_{\mathrm{ncvx}}$
is extremely close to the solution $\bm{L}_{\mathrm{cvx}}$ to the
convex problem \eqref{eq:sqrt-mc}.
\end{enumerate}
Combining the two key steps via triangle inequality finishes the proof. 

\paragraph*{Step 1: Nonconvex optimization.}

The nonconvex optimization problem \eqref{eq:nonconvex-opt-theta}
has two groups of decision variables, i.e., $(\bm{X},\bm{Y})$ and
$\theta$. Also note that given a fixed pair $(\bm{X},\bm{Y})$, the
optimal choice of $\theta$ is simply given by $\theta=\|\mathcal{P}_{\Omega}(\bm{X}\bm{Y}^{\top}-\bm{M})\|_{\mathrm{F}}$.
Therefore it is natural to consider an alternating minimization method
to construct an approximate stationary point of the nonconvex program
\eqref{eq:nonconvex-opt-theta}; see Algorithm \ref{algo:non-cvx-gd}.
Given a current iterate $(\bm{X}_{t},\bm{Y}_{t},\theta_{t})$, the
algorithm first runs one step of gradient descent on $(\bm{X},\bm{Y})$
while fixing $\theta_{t}$. It then updates $\theta_{t+1}=\|\mathcal{P}_{\Omega}(\bm{X}_{t+1}\bm{Y}_{t+1}^{\top}-\bm{M})\|_{\mathrm{F}}$
to be the optimal choice given the new iterate $(\bm{X}_{t+1},\bm{Y}_{t+1})$.
In the end, Algorithm \ref{algo:non-cvx-gd} returns the point $\bm{L}_{\mathrm{ncvx}}$
with the smallest gradient among the iterates as an approximate stationary
point. 

\begin{algorithm}[t]
\caption{\label{algo:non-cvx-gd}Gradient descent on the nonconvex formulation
of square root matrix completion}

\textbf{Input:} initialization$\bm{X}_{0}=\bm{X}^{\star},\bm{Y}_{0}=\bm{Y}^{\star},\theta_{0}=\|\mathcal{P}_{\Omega}(\bm{X}^{\star}\bm{Y}^{\star\top}-\bm{M})\|_{\mathrm{F}}$,
step size $\eta\asymp\sigma/(\sqrt{p}\kappa^{3}\sigma_{\max})$, and
total number of iterations $t_{0}=n^{18}$.

\textbf{Gradient updates: for $t=0,1,\ldots,t_{0}-1$ do} \begin{subequations}

\begin{align}
\bm{X}_{t+1} & =\bm{X}_{t}-\eta\nabla_{\bm{X}}f(\bm{X}_{t},\bm{Y}_{t},\theta_{t})=\bm{X}_{t}-\eta\left(\tfrac{1}{\theta_{t}}\mathcal{P}_{\Omega}(\bm{X}_{t}\bm{Y}_{t}^{\top}-\bm{M})\bm{Y}_{t}+\lambda\bm{X}_{t}\right);\label{eq:GD-X}\\
\bm{Y}_{t+1} & =\bm{Y}_{t}-\eta\nabla_{\bm{Y}}f(\bm{X}_{t},\bm{Y}_{t},\theta_{t})=\bm{Y}_{t}-\eta\left(\tfrac{1}{\theta_{t}}\left[\mathcal{P}_{\Omega}(\bm{X}_{t}\bm{Y}_{t}^{\top}-\bm{M})\right]^{\top}\bm{X}_{t}+\lambda\bm{Y}_{t}\right);\label{eq:GD-Y}\\
\theta_{t+1} & =\|\mathcal{P}_{\Omega}(\bm{X}_{t+1}\bm{Y}_{t+1}^{\top}-\bm{M})\|_{\mathrm{F}}.\label{eq:GD-theta}
\end{align}

\end{subequations}

\textbf{Define
\[
t^{\star}\coloneqq\arg\min_{0\le t\le t_{0}}\|\nabla_{\bm{X},\bm{Y}}f(\bm{X}_{t},\bm{Y}_{t},\theta_{t})\|_{\mathrm{F}},
\]
where 
\[
\nabla_{\bm{X},\bm{Y}}f(\bm{X}_{t},\bm{Y}_{t},\theta_{t})=\begin{bmatrix}\frac{1}{\theta_{t}}\mathcal{P}_{\Omega}(\bm{X}_{t}\bm{Y}_{t}^{\top}-\bm{M})\bm{Y}_{t}+\lambda\bm{X}_{t}\\
\frac{1}{\theta_{t}}\left[\mathcal{P}_{\Omega}(\bm{X}_{t}\bm{Y}_{t}^{\top}-\bm{M})\right]^{\top}\bm{X}_{t}+\lambda\bm{Y}_{t}
\end{bmatrix}.
\]
}

\textbf{Output:} $\bm{L}_{\mathrm{ncvx}}\coloneqq\bm{X}_{t^{\star}}\bm{Y}_{t^{\star}}^{\top}$, $\bm{X}_{\mathrm{ncvx}} \coloneqq \bm{X}_{t^{\star}}$, and $\bm{Y}_{\mathrm{ncvx}} \coloneqq \bm{Y}_{t^{\star}}$.
\end{algorithm}

The following lemma ensures that $\bm{L}_{\mathrm{ncvx}}$ is an approximate
stationary point of the nonconvex problem and more importantly is
close to the groundtruth matrix $\bm{L}^{\star}$. The proof is deferred
to Section \ref{subsec:Proof-of-Lemma-ncvx-XY-L}. 

\begin{lemma}\label{lemma:ncvx-iterates-bound-XY-L}

Instate the assumptions of Theorem~\ref{thm:main}. With probability
at least $1-O(n^{-3})$, one has 

\begin{subequations}

\begin{align}
\|\bm{L}_{\mathrm{ncvx}}-\bm{L}^{\star}\|_{\mathrm{F}} & \le3\kappa C_{\mathrm{F}}\left(\frac{\sigma}{\sigma_{\min}}\sqrt{\frac{n}{p}}\right)\|\bm{L}^{\star}\|_{\mathrm{F}},\label{eq:ncvx-iterates-bound-XY-L-fro}\\
\|\bm{L}_{\mathrm{ncvx}}-\bm{L}^{\star}\|_{\mathrm{\infty}} & \le3\sqrt{\kappa^{3}\mu r}C_{\mathrm{\infty}}\left(\frac{\sigma}{\sigma_{\min}}\sqrt{\frac{n\log n}{p}}\right)\|\bm{L}^{\star}\|_{\mathrm{\infty}},\label{eq:ncvx-iterates-bound-XY-L-infty}\\
\|\bm{L}_{\mathrm{ncvx}}-\bm{L}^{\star}\| & \le3C_{\mathrm{op}}\left(\frac{\sigma}{\sigma_{\min}}\sqrt{\frac{n}{p}}\right)\|\bm{L}^{\star}\|,\label{eq:ncvx-iterates-bound-XY-L-spec}
\end{align}
where $C_{\mathrm{F}},C_{\infty},C_{\mathsf{op}}$ are three universal
positive constants. 

\end{subequations}

\end{lemma}

\paragraph*{Step 2: Bridging convex and nonconvex solutions. }

It remains to show that $\bm{L}_{\mathrm{ncvx}}$ is extremely close
to the convex solution $\bm{L}_{\mathrm{cvx}}$, which is provided
in the following lemma. 

\begin{lemma}\label{lemma:ncvx-cvx}Instate the assumptions of Theorem~\ref{thm:main}.
With probability exceeding $1-O(n^{-3})$, one has 
\[
\left\Vert \bm{L}_{\mathrm{ncvx}}-\bm{L}_{\mathrm{cvx}}\right\Vert _{\mathrm{F}}\le\frac{1}{n^{5}}\frac{\lambda\sigma}{\sigma_{\min}}\|\bm{L}^{\star}\|_{\mathrm{F}}.
\]

\end{lemma}

\noindent See Section \ref{sec:Connection-cvx-ncvx} for the proof
of this lemma. 

\medskip

We remark in passing that the polynomial factor $n^{-5}$ in Lemma~\ref{lemma:ncvx-cvx}
is arbitrarily chosen, and the exponent~$5$ can be replaced with
any large constant. The essence is that the difference between $\bm{L}_{\mathrm{ncvx}}$
and $\bm{L}_{\mathrm{cvx}}$ is orderwise much smaller compared to
the estimation error of $\bm{L}_{\mathrm{ncvx}}$ itself.
Such proximity between $\bm{L}_{\mathrm{ncvx}}$ and $\bm{L}_{\mathrm{cvx}}$
is verified empirically in Figure~\ref{fig:ncvx-cvx}.

\medskip

\begin{figure}
\centering{}\includegraphics[width=0.45\textwidth]{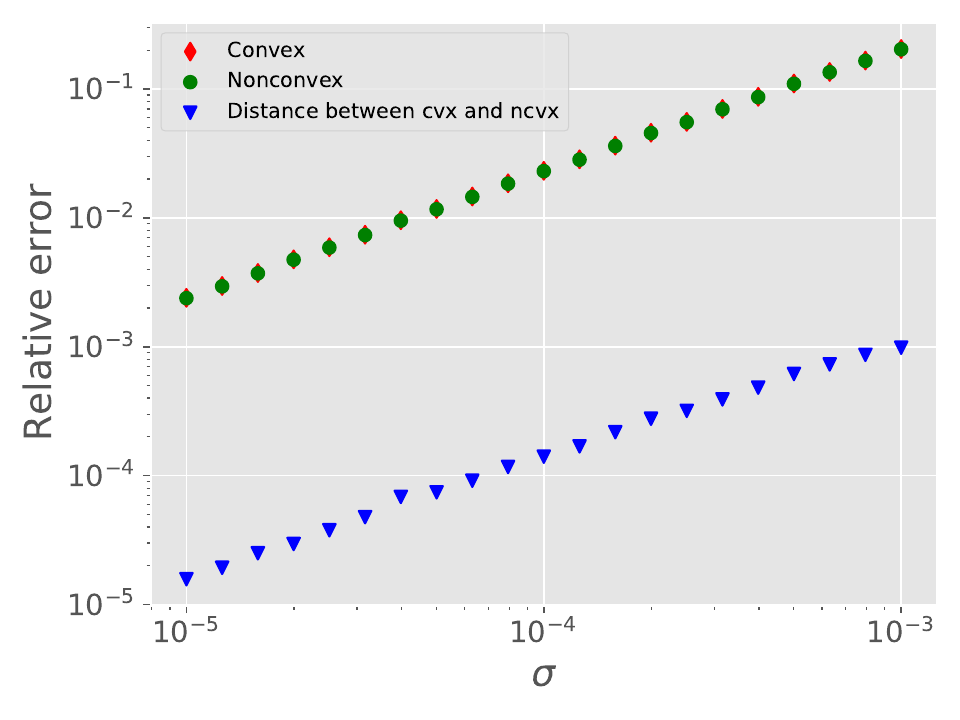}\caption{\label{fig:ncvx-cvx}Relative Frobenius estimation error of 
convex and nonconvex solutions and their distance. The parameters are
chosen as: $n=200,r=5,p=0.5$ while $\sigma$ varies from $10^{-5}$ to $10^{-3}$.}
\end{figure}

Now we are ready to combine the previous two steps and finish the
proof of Theorem~\ref{thm:main}.

\paragraph{Proof of Theorem~\ref{thm:main}.}

Combine Lemmas~\ref{lemma:ncvx-iterates-bound-XY-L}-\ref{lemma:ncvx-cvx}
with the triangle inequality to arrive at 
\begin{align*}
\left\Vert \bm{L}_{\mathrm{cvx}}-\bm{L}^{\star}\right\Vert _{\mathrm{F}} & \le\left\Vert \bm{L}_{\mathrm{ncvx}}-\bm{L}_{\mathrm{cvx}}\right\Vert _{\mathrm{F}}+\left\Vert \bm{L}_{\mathrm{ncvx}}-\bm{L}^{\star}\right\Vert _{\mathrm{F}}\\
 & \le\left[\frac{1}{n^{5}}\frac{\lambda\sigma}{\sigma_{\min}}+3\kappa C_{\mathrm{F}}\left(\frac{\sigma}{\sigma_{\min}}\sqrt{\frac{n}{p}}\right)\right]\|\bm{L}^{\star}\|_{\mathrm{F}}\\
 & \le4\kappa C_{\mathrm{F}}\left(\frac{\sigma}{\sigma_{\min}}\sqrt{\frac{n}{p}}\right)\|\bm{L}^{\star}\|_{\mathrm{F}},
\end{align*}
where the last relation uses the facts that $\lambda\asymp1/\sqrt{n}$
and that $p\gtrsim1/\sqrt{n}$. Redefine $4C_{\mathrm{F}}$ to be
$C_{\mathrm{F}}$ to complete the proof of the bound~\eqref{eq:L-bound-fro}.
The other two bounds on the operator norm and the $\ell_{\infty}$
norm follow from similar arguments. We omit here for brevity. 

\subsection{Proof of Lemma~\ref{lemma:ncvx-iterates-bound-XY-L} \label{subsec:Proof-of-Lemma-ncvx-XY-L}}

Since Algorithm~\ref{algo:non-cvx-gd} operates in the space of low-rank
factors, we start with establishing guarantees for the stacked low-rank
factor $\bm{F}_{t}\coloneqq\begin{bmatrix}\bm{X}_{t}\\
\bm{Y}_{t}
\end{bmatrix}\in\mathbb{R}^{2n\times r}$, and then translate the guarantees to the matrix space $\bm{L}_{t}=\bm{X}_{t}\bm{Y}_{t}^{\top}$.
Special care is needed as the decomposition $\bm{L}=\bm{X}\bm{Y}^{\top}$
is not unique in $(\bm{X},\bm{Y})$, and hence we need to account
for the rotational ambiguity in $(\bm{X},\bm{Y})$. To this end, for
each $t\geq0$, we define the optimal rotation matrix to be 
\begin{equation}
\bm{H}_{t}\coloneqq\mathrm{argmin}_{\bm{R}\in\mathcal{O}^{r\times r}}\quad\|\bm{X}_{t}\bm{R}-\bm{X}^{\star}\|_{\mathrm{F}}^{2}+\|\bm{Y}_{t}\bm{R}-\bm{Y}^{\star}\|_{\mathrm{F}}^{2}.\label{eq:def-Ht}
\end{equation}

\paragraph{Introducing leave-one-out sequences. }

In order to control the $\ell_{2,\infty}$ error of $\bm{F}_{t}$
(and hence $\ell_{\infty}$ error of $\bm{L}_{t}$), we construct
$2n$ leave-one-out auxiliary sequences $\{\bm{F}_{t}^{(l)}\}_{1\leq l\leq2n,t\geq0}$.
The hope is that $\{\bm{F}_{t}^{(l)}\}_{1\leq l\leq2n,t\geq0}$ serves
as a good approximation to the original sequence $\{\bm{F}_{t}\}_{t\geq0}$,
while at the same time is more amenable to statistical analysis. 

To formally construct such leave-one-out sequences, we first define
$2n$ auxiliary loss functions. For each $1\le l\le n$, define 
\[
f^{(l)}(\bm{X},\bm{Y},\theta)=\frac{1}{2}\left(\frac{\|\mathcal{P}_{\Omega_{-l,\cdot}}(\bm{L}-\bm{M})\|_{\mathrm{F}}^{2}+p\|\mathcal{P}_{l,\cdot}(\bm{L}-\bm{M})\|_{\mathrm{F}}^{2}}{\theta}+\theta\right)+\frac{\lambda}{2}\left(\|\bm{X}\|_{\mathrm{F}}^{2}+\|\bm{Y}\|_{\mathrm{F}}^{2}\right)
\]
where 
\[
\left[\mathcal{P}_{\Omega_{-l,\cdot}}(\bm{B})\right]_{ij}=\begin{cases}
B_{ij}, & \text{if }(i,j)\in\Omega\text{ and }i\neq l\\
0, & \text{otherwise}
\end{cases},\quad\text{and}\quad\left[\mathcal{P}_{l,\cdot}(\bm{B})\right]_{ij}=\begin{cases}
B_{ij}, & \text{if }i=l\\
0, & \text{otherwise}
\end{cases}.
\]
Similarly, for each $n+1\le l\le2n$, we define 
\[
f^{(l)}(\bm{X},\bm{Y},\theta)=\frac{1}{2}\left(\frac{\|\mathcal{P}_{\Omega_{\cdot,-(l-n)}}(\bm{L}-\bm{M})\|_{\mathrm{F}}^{2}+p\|\mathcal{P}_{l,\cdot}(\bm{L}-\bm{M})\|_{\mathrm{F}}^{2}}{\theta}+\theta\right)+\frac{\lambda}{2}\left(\|\bm{X}\|_{\mathrm{F}}^{2}+\|\bm{Y}\|_{\mathrm{F}}^{2}\right)
\]
where 
\[
\left[\mathcal{P}_{\Omega_{\cdot,-(l-n)}}(\bm{B})\right]_{ij}=\begin{cases}
B_{ij}, & \text{if }(i,j)\in\Omega\text{ and }j\neq l-n\\
0, & \text{otherwise}
\end{cases}\quad\text{and}\quad\left[\mathcal{P}_{l,\cdot}(\bm{B})\right]_{ij}=\begin{cases}
B_{ij}, & \text{if }j=l-n\\
0, & \text{otherwise}
\end{cases}.
\]
With these notations in place, Algorithm~\ref{algo:leave-one-out}
details the way we construct the leave-one-out sequences. 

\begin{algorithm}
\caption{Gradient descent generating the leave-one-out sequences \label{algo:leave-one-out}}

\textbf{Initialization:} $\bm{X}_{0}^{(l)}=\bm{X}^{\star},\bm{Y}_{0}^{(l)}=\bm{Y}^{\star},\theta_{0}^{(l)}=\|\mathcal{P}_{\Omega}(\bm{X}^{\star}\bm{Y}^{\star\top}-\bm{M})\|_{\mathrm{F}}$,
step size $\eta\asymp\sigma/(\sqrt{p}\kappa^{3}\sigma_{\max})$, and
total number of iterations $t_{0}=n^{18}$.

\textbf{Gradient updates: for $t=0,1,\cdots,t_{0}-1$ do }

\begin{subequations}

\begin{align}
\bm{X}_{t+1}^{(l)} & =\bm{X}_{t}^{(l)}-\eta\nabla_{\bm{X}}f^{(l)}(\bm{X}_{t}^{(l)},\bm{Y}_{t}^{(l)},\theta_{t})\label{eq:GD-leave-one-out-X-1};\\
\bm{Y}_{t+1}^{(l)} & =\bm{Y}_{t}^{(l)}-\eta\nabla_{\bm{Y}}f^{(l)}(\bm{X}_{t}^{(l)},\bm{Y}_{t}^{(l)},\theta_{t})\label{eq:GD-Y-1};\\
\theta_{t+1} & =\left\Vert \mathcal{P}_{\Omega}\left(\bm{X}_{t+1}\bm{Y}_{t+1}^{\top}-\bm{M}\right)\right\Vert _{\mathrm{F}}.\label{eq:GD-theta-1}
\end{align}

\end{subequations}
\end{algorithm}

Similar constructions have been deployed in the papers \cite{chen2020noisy}
and \cite{chen2021bridging}. However, it is worth pointing out that
the sequence $\{\theta_{t}\}$ is produced according to the original
sequence, instead of the leave-one-out sequence. This change is tailored
to the analysis of the \est~estimator as it aligns better with the
original loss function $f$, while allowing us to reuse several keys
results in the paper \cite{chen2020noisy}. 

\paragraph{Properties of the iterates.}

As planned, we aim to show that the leave-one-out iterates $\{\bm{F}_{t}^{(l)}\}_{1\leq l\leq2n,t\geq0}$
stay extremely close to the original iterates $\{\bm{F}_{t}\}_{t\geq0}$,
and that $\{\bm{F}_{t}\}_{t\geq0}$ is close to the groundtruth factor
$\bm{F}^{\star}$. Such properties are collected in the following
lemma. 

\begin{lemma}\label{lemma:induction} With probability at least $1-O(n^{-3})$,
the following statements hold for all iterations $0\le t\le t_{0}$:
\begin{subequations}\label{eq:ncvx-iterates-full}

\begin{align}
\|\bm{F}_{t}\bm{H}_{t}-\bm{F}^{\star}\|_{\mathrm{F}} & \le C_{\mathrm{F}}\frac{\sigma}{\sigma_{\min}}\sqrt{\frac{n}{p}}\|\bm{X}^{\star}\|_{\mathrm{F}},\label{eq:ncvx-itr-bound-fro}\\
\|\bm{F}_{t}\bm{H}_{t}-\bm{F}^{\star}\| & \le C_{\mathrm{op}}\frac{\sigma}{\sigma_{\min}}\sqrt{\frac{n}{p}}\|\bm{X}^{\star}\|,\label{eq:ncvx-itr-bound-spectral}\\
\max_{1\le l\le2n}\|\bm{F}_{t}\bm{H}_{t}-\bm{F}_{t}^{(l)}\bm{R}_{t}^{(l)}\|_{\mathrm{F}} & \le C_{\mathrm{3}}\frac{\sigma}{\sigma_{\min}}\sqrt{\frac{n\log n}{p}}\|\bm{F}^{\star}\|_{2,\infty},\label{eq:ncvx-itr-bound-leave1out}\\
\max_{1\le l\le2n}\|(\bm{F}_{t}^{(l)}\bm{H}_{t}^{(l)}-\bm{F}^{\star})_{l,\cdot}\|_{\mathrm{F}} & \le C_{\mathrm{4}}\kappa\frac{\sigma}{\sigma_{\min}}\sqrt{\frac{n\log n}{p}}\|\bm{F}^{\star}\|_{2,\infty},\label{eq:ncvx-itr-bound-leave1out-to-ground-truth}\\
\|\bm{F}_{t}\bm{H}_{t}-\bm{F}^{\star}\|_{2,\infty} & \le C_{\mathrm{\infty}}\kappa\frac{\sigma}{\sigma_{\min}}\sqrt{\frac{n\log n}{p}}\|\bm{F}^{\star}\|_{2,\infty},\label{eq:ncvx-itr-bound-2infty}
\end{align}

\end{subequations}

for some positive constants $C_{\mathrm{F}},C_{\mathrm{op}},C_{3},C_{4},C_{\infty}$.
Here $\bm{H}_{t}^{(l)}$ and $\bm{R}_{t}^{(l)}$ are defined as 

\begin{align*}
\bm{H}_{t}^{(l)} & \coloneqq\mathrm{argmin}_{\bm{R}\in\mathcal{O}^{r\times r}}\|\bm{F}_{t}^{(l)}\bm{R}-\bm{F}^{\star}\|_{\mathrm{F}}\\
\bm{R}_{t}^{(l)} & \coloneqq\mathrm{argmin}_{\bm{R}\in\mathcal{O}^{r\times r}}\|\bm{F}_{t}^{(l)}\bm{R}-\bm{F}_{t}\bm{H}_{t}\|_{\mathrm{F}}.
\end{align*}
Furthermore the output $(\bm{X}_{t^{\star}}, \bm{Y}_{t^{\star}})$ has small gradient: 
\begin{equation}
\|\nabla_{\bm{X},\bm{Y}}f(\bm{X}_{t^{\star}},\bm{Y}_{t^{\star}},\theta_{t^{\star}})\|_{\mathrm{F}}\le C_{\mathrm{grad}}\frac{1}{n^{8}}\sqrt{\frac{\sigma_{\max}}{p}}. \label{eq:small-grad}
\end{equation}

\end{lemma}

\noindent See Section~\ref{sec:proof-induction} for the proof of this lemma. 
\medskip

% \begin{remark}

% This lemma shows that the gradient descent iterates are close to the
% ground truth in various metrics and they contain an approximate stationary
% point that has very small gradient size. 

% \end{remark}

Now we are ready to prove Lemma~\ref{lemma:ncvx-iterates-bound-XY-L}
based on the results presented in Lemma~\ref{lemma:induction}. 

\paragraph{Proof of Lemma~\ref{lemma:ncvx-iterates-bound-XY-L}.}

By the triangle inequality, one has
\begin{align*}
\|\bm{X}_{t^{\star}}\bm{Y}_{t^{\star}}^{\top}-\bm{L}^{\star}\| & \le\|\bm{X}_{t^{\star}}\bm{Y}_{t^{\star}}^{\top}-\bm{X}_{t^{\star}}\bm{Y}^{\star\top}\|+\|\bm{X}_{t^{\star}}\bm{Y}^{\star\top}-\bm{L}^{\star}\|\\
 & \le\|\bm{Y}_{t^{\star}}-\bm{Y}^{\star}\|\|\bm{X}_{t^{\star}}\|+\|\bm{X}_{t^{\star}}-\bm{X}^{\star}\|\|\bm{Y}^{\star}\|.
\end{align*}
Use relation~\eqref{eq:ncvx-itr-bound-spectral} to obtain 
\begin{align*}
\|\bm{X}_{t^{\star}}\bm{Y}_{t^{\star}}^{\top}-\bm{L}^{\star}\| & \le3C_{\mathrm{op}}\frac{\sigma}{\sigma_{\min}}\sqrt{\frac{n}{p}}\|\bm{X}^{\star}\|\|\bm{X}^{\star}\|=3C_{\mathrm{op}}\frac{\sigma}{\sigma_{\min}}\sqrt{\frac{n}{p}}\|\bm{L}^{\star}\|.
\end{align*}
The first inequality uses $\|\bm{X}_{t^{\star}}\|\le2\|\bm{X}^{\star}\|$,
which is a direct consequence of \eqref{eq:ncvx-itr-bound-spectral}
and the last line uses $\|\bm{L}^{\star}\|=\sigma_{\max}=\|\bm{X}^{\star}\|^{2}$.
Similarly we have 
\[
\|\bm{X}_{t^{\star}}\bm{Y}_{t^{\star}}^{\top}-\bm{L}^{\star}\|_{\mathrm{F}}\le3C_{\mathrm{\mathrm{F}}}\frac{\sigma}{\sigma_{\min}}\sqrt{\frac{n}{p}}\|\bm{X}^{\star}\|_{\mathrm{F}}\|\bm{X}^{\star}\|\overset{(\text{i})}{\le}3\kappa C_{\mathrm{op}}\frac{\sigma}{\sigma_{\min}}\sqrt{\frac{n}{p}}\|\bm{L}^{\star}\|_{\mathrm{F}}
\]
and
\[
\|\bm{X}_{t^{\star}}\bm{Y}_{t^{\star}}^{\top}-\bm{L}^{\star}\|_{\infty}\le3C_{\mathrm{\infty}}\frac{\sigma}{\sigma_{\min}}\sqrt{\frac{n\log n}{p}}\|\bm{F}^{\star}\|_{2,\infty}\|\bm{F}^{\star}\|_{2,\infty}\overset{(\text{ii})}{\le}3\sqrt{\kappa^{3}\mu r}C_{\infty}\frac{\sigma}{\sigma_{\min}}\sqrt{\frac{n\log n}{p}}\|\bm{L}^{\star}\|_{\infty}.
\]
Here step (i) uses the fact$\|\bm{X}^{\star}\|_{\mathrm{F}}\|\bm{X}^{\star}\|\le\kappa\|\bm{L}^{\star}\|_{\mathrm{F}}$, 
whereas in step (ii) we use 
\[\|\bm{F}^{\star}\|_{2,\infty}\|\bm{F}^{\star}\|_{2,\infty}\le\sqrt{\kappa^{3}\mu r}\|\bm{L}^{\star}\|_{\infty}.\]

\subsection{Proof of Lemma~\ref{lemma:ncvx-cvx} \label{sec:Connection-cvx-ncvx}}

Before embarking on the main proof, we state a few useful properties
of the noise matrix $\bm{E}$ and the nonconvex solution $\bm{L}_{\mathrm{ncvx}}$.
These properties allow us to establish the proximity between the approximate
stationary point $\bm{L}_{\mathrm{ncvx}}$ and the convex solution
$\bm{L}_{\mathrm{cvx}}$. 

The first property is concerned with the size of the regularization
parameter, which appeared as Lemma~3 in the paper~\cite{chen2020noisy}.

\begin{lemma}\label{lemma:E-bound}Suppose that $n^{2}p\ge Cn\log^{2}n$
for some sufficiently large constant $C>0$. Take $\lambda=C_{\lambda}/\sqrt{n}$
for some absolute constant $C_{\lambda}$. Then with probability at least $1-O(n^{-10})$,
one has 
\begin{equation}
\|\mathcal{P}_{\Omega}(\bm{E})\|\le\frac{\lambda}{16}np^{1/2}\sigma.\label{eq:noise-bound}
\end{equation}
\end{lemma}

The next property is on the injectivity of $\mathcal{P}_{\Omega}$
in the tangent space $T$ at $\bm{L}_{\mathrm{ncvx}}$. More precisely,
letting $\bm{U\Sigma V}^{\top}$ be the SVD of $\bm{L}_{\mathrm{ncvx}}$,
we define the tangent space $T$ at $\bm{L}_{\mathrm{ncvx}}$ as 

\[
T=\left\{ \boldsymbol{U}\boldsymbol{A}^{\top}+\boldsymbol{B}\boldsymbol{V}^{\top}\mid\boldsymbol{A},\ensuremath{\bm{B}}\ensuremath{\in\mathbb{R}^{n\times r}}\right\} .
\]

\begin{lemma}\label{lemma:injectivity} Instate the assumptions of
Theorem~\ref{thm:main}. With probability exceeding $1-O(n^{-3})$, 
for all $\bm{H}\in T$
\begin{equation}
p^{-1/2}\|\mathcal{P}_{\Omega}(\bm{H})\|_{\mathrm{F}}\ge C_{\mathrm{inj}}\|\bm{H}\|_{\mathrm{F}},\qquad \text{where}\quad C_{\mathrm{inj}}=(32\kappa)^{-1/2}.\label{eq:injectivity}
\end{equation}

\end{lemma}

\begin{proof}
This is an easy consequence of Lemma 4 in the paper \cite{chen2020noisy} and the relation~\eqref{eq:ncvx-itr-bound-2infty}.
\end{proof} 

Last but not least, the lemma collects several interesting properties of the nonconvex solution $\bm{L}_{\mathrm{ncvx}}$, as well as its low-rank factors $\bm{X}_{\mathrm{ncvx}}, \bm{Y}_{\mathrm{ncvx}}$.

\begin{lemma}\label{lemma:error-size}

The approximate stationary point $\bm{L}_{\mathrm{ncvx}}$ satisfies
\begin{subequations}
\begin{align}
\sqrt{\sigma_{\min}/2}\leq\sigma_{\min}(\bm{X}_{\mathrm{ncvx}})\leq\sigma_{\max}(\bm{X}_{\mathrm{ncvx}}) & \leq\sqrt{2\sigma_{\max}};\\
\sqrt{\sigma_{\min}/2}\leq\sigma_{\min}(\bm{Y}_{\mathrm{ncvx}})\leq\sigma_{\max}(\bm{Y}_{\mathrm{ncvx}}) & \leq\sqrt{2\sigma_{\max}};\\
\frac{1}{2}np^{1/2}\sigma\le\|\mathcal{P}_{\Omega}(\bm{L}_{\mathrm{ncvx}}-\bm{M})\|_{\mathrm{F}} & \le2np^{1/2}\sigma;\label{eq:theta-bound-Lncvx}\\
\|\mathcal{P}_{\Omega}(\bm{X}\bm{Y}^{\top}-\bm{L}^{\star})-p(\bm{X}\bm{Y}^{\top}-\bm{L}^{\star})\| & \le\frac{\lambda}{16}np^{1/2}\sigma.\label{eq:regularization-req-P}
\end{align}

\end{subequations}\end{lemma} 

\noindent See Section~\ref{subsec:proof-lemma-error-size} for the proof of this lemma. 
\medskip

% \begin{remark}The lower bound of $\|\mathcal{P}_{\Omega}(\bm{L}_{\mathrm{ncvx}}-\bm{M})\|_{\mathrm{F}}$
% in \eqref{eq:theta-bound-Lncvx} is crucial here. One can consider $\|\mathcal{P}_{\Omega}(\bm{L}_{\mathrm{ncvx}}-\bm{M})\|_{\mathrm{F}}$
% as an inherent rescaling factor of the square error $\|\mathcal{P}_{\Omega}(\bm{L}_{\mathrm{ncvx}}-\bm{M})\|_{\mathrm{F}}^{2}$
% and its subgradients. In this view, Equation~\eqref{eq:theta-bound}
% ensures that $\|\mathcal{P}_{\Omega}(\bm{L}_{\mathrm{ncvx}}-\bm{M})\|_{\mathrm{F}}\asymp np^{1/2}\sigma$
% gives a correct scaling that removes the dependency of $\sigma$ from
% the subgradients. A significantly smaller $\|\mathcal{P}_{\Omega}(\bm{L}_{\mathrm{ncvx}}-\bm{M})\|_{\mathrm{F}}$
% would result in effectively less regularization and a gap between
% $\bm{L}_{\mathrm{ncvx}}$ and $\bm{L}_{\mathrm{cvx}}$.\end{remark}

For notational simplicity, we define 
\[
g(\bm{X},\bm{Y})\coloneqq f(\boldsymbol{X},\boldsymbol{Y},\|\mathcal{P}_{\Omega}(\bm{X}\bm{Y}^{\top}-\bm{M})\|_{\mathrm{F}})=\|\mathcal{P}_{\Omega}(\bm{X}\bm{Y}^{\top}-\bm{M})\|_{\mathrm{F}}+\frac{\lambda}{2}\left(\|\bm{X}\|_{\mathrm{F}}^{2}+\|\bm{Y}\|_{\mathrm{F}}^{2}\right).
\]
In other words, $g(\bm{X},\bm{Y})$ is the minimal value of $f(\bm{X},\bm{Y},\theta)$
when $(\bm{X},\bm{Y})$ is fixed. 

Now we are ready to present the key lemma of this section, which relates
the difference between $\bm{L}_{\mathrm{ncvx}}$ and $\bm{L}_{\mathrm{cvx}}$
to the size of the gradient $\nabla g(\bm{X}_{\mathrm{ncvx}},\bm{Y}_{\mathrm{ncvx}})$.
The proof is deferred to Section~\ref{subsec:Proof-of-Lemma-ncvx-cvx-full}. 

\begin{lemma}\label{lemma:ncvx-to-cvx-full}

Suppose that $(\bm{X}_{\mathrm{ncvx}},\bm{Y}_{\mathrm{ncvx}})$ has small gradient
in the sense that 
\begin{equation}
\left\Vert \nabla g(\bm{X}_{\mathrm{ncvx}},\bm{Y}_{\mathrm{ncvx}})\right\Vert _{\mathrm{F}}\le\frac{\sqrt{\sigma_{\min}}}{280\kappa}\max\left\{ C_{\mathrm{inj}}\sqrt{p},\frac{1}{2}\lambda^{2}n\sigma\right\} .\label{eq:gradient-size-bound}
\end{equation}
Then on the event that Lemmas~\ref{lemma:E-bound}-\ref{lemma:error-size} hold, 
any minimizer \textbf{$\bm{L}_{\mathrm{cvx}}$} of the convex
program~\eqref{eq:sqrt-mc}\textbf{ }satisfies 
\[
\left\Vert \bm{L}_{\mathrm{ncvx}}-\bm{L}_{\mathrm{cvx}}\right\Vert _{\mathrm{F}}\le\frac{\lambda\kappa^{2}}{\sqrt{p\sigma_{\min}}}n\sigma\|\nabla g(\bm{X}_{\mathrm{ncvx}},\bm{Y}_{\mathrm{ncvx}})\|_{\mathrm{F}}.
\]

\end{lemma}

\begin{remark}

Observe that if $\|\nabla g(\bm{X}_{\mathrm{ncvx}},\bm{Y}_{\mathrm{ncvx}})\|_{\mathrm{F}}=0$, i.e.,
if $\bm{L}_{\mathrm{ncvx}}$ is an exact stationary point of the nonconvex
\est~problem, $\bm{L}_{\mathrm{ncvx}}$ is also a solution to the
convex problem \eqref{eq:sqrt-mc}. 

\end{remark}

With the help of Lemma~\ref{lemma:ncvx-to-cvx-full}, we can prove
Lemma \ref{lemma:ncvx-cvx} now. 

\paragraph{Proof of Lemma~\ref{lemma:ncvx-cvx}.} 

First, Lemma~\ref{lemma:induction} tells us that the nonconvex solution
$(\bm{X}_{t^{\star}},\bm{Y}_{t^{\star}})$ satisfies the bound~\eqref{eq:gradient-size-bound}
on the size of the gradient. This together with Lemmas~\ref{lemma:E-bound}
to \ref{lemma:error-size} allows us to invoke Lemma~\ref{lemma:ncvx-to-cvx-full}
to obtain
\begin{align*}
\left\Vert \bm{L}_{\mathrm{ncvx}}-\bm{L}_{\mathrm{cvx}}\right\Vert _{\mathrm{F}} & \lesssim\frac{\lambda\kappa^{2}}{\sqrt{p\sigma_{\min}}}n\sigma\|\nabla g(\bm{X}_{\mathrm{ncvx}},\bm{Y}_{\mathrm{ncvx}})\|_{\mathrm{F}} \lesssim\frac{1}{n^{5}}\frac{\lambda\sigma}{\sigma_{\min}}\|\bm{L}^{\star}\|_{\mathrm{F}},
\end{align*}
where the last inequality uses the gradient upper bound~\eqref{eq:gradient-size-bound}, $\|\bm{L}^{\star}\|_{\mathrm{F}}\ge\|\bm{L}^{\star}\|\ge\sigma_{\max}=\kappa\sigma_{\min}$,
and the fact that the sample size assumption $n^{2}p\ge C_{\mathrm{sample}}\kappa^{4}\mu^{2}r^{2}n\log^{3}n$
implies $np\gtrsim1$ and $\kappa\lesssim n$.

\section{Simulation}

In this section, we further illustrate the performance of the tuning-free
square root matrix completion through two sets of comparative simulation studies.
First we compare the performance of \est~to the non-sqaure-root estimator \eqref{eq:mc}
with oracle and cross-validated parameters. This allows
us to examine whether we sacrifice a significant amount of performance
in achieving the tuning-free property. Second, we do the same comparison on approximately low rank matrices. This helps us understand how robust the estimator is against misspecified low-rank assumption.

\paragraph*{Comparing \est~with standard approach \eqref{eq:mc}.}

For the non-square-root approach \eqref{eq:mc}, as the sampling probability $p$ and
noise level $\sigma$ is unknown, the regularization parameter needs
to be carefully chosen. Here we compare \est~with \eqref{eq:mc} using oracle
and $k$-fold cross-validated regularization parameters, namely 
\[
\lambda_{\mathrm{oracle}}\coloneqq\arg\min_{\lambda}\quad\left\Vert \bm{L}^{\star}-\hat{\bm{L}}_{\lambda,\Omega}\right\Vert _{\mathrm{F}},
\]
\[
\lambda_{\mathrm{CV}}\coloneqq\arg\min_{\lambda}\quad\sum_{i=1}^{k}\left\Vert \mathcal{P}_{\Omega_{i}}\left(\bm{L}-\hat{\bm{L}}_{\lambda,\Omega_{-i}}\right)\right\Vert _{\mathrm{F}}^{2},
\]
where 
\[
\hat{\bm{L}}_{\lambda,\Omega}\coloneqq\arg\min_{\bm{L}\in\mathbb{R}^{n\times n}}\quad\sum_{(i,j)\in\Omega}(L_{ij}-M_{ij})^{2}+\lambda\|\bm{L}\|_{*}
\]
with $\Omega_{i}$ being the $i$-th fold of the sampled entries and $\Omega_{-i}\coloneqq\Omega\setminus\Omega_{i}$
. Due to computational limit, our experiment uses estimates $\hat{\lambda}_{\mathrm{oracle}},\hat{\lambda}_{\mathrm{CV}}$
obtained by taking minimum over a discrete set of parameters that is close
to the true $\lambda_{\mathrm{oracle}}$. In practice $\lambda_{\mathrm{oracle}}$ is inaccessible
as we do not know $\bm{L}^{\star}$. Meanwhile $\lambda_{\mathrm{CV}}$
takes $k\cdot n_{\lambda}$ runs of an algorithm for \eqref{eq:mc} to obtain,
where $n_{\lambda}$ is the number of $\lambda$'s one tries in cross-validation. This can be computationally prohibitive when the matrices of interest have different sampling rate $p$ and noise level $\sigma$, in which case cross-validation is needed for each matrix in order to get a reasonable $\lambda$. In comparison, the tuning-free property of \est~makes the regularization parameter much easier to obtain.

\begin{figure}
\begin{minipage}[t]{0.45\columnwidth}%
\begin{center}
\includegraphics[width=1\textwidth]{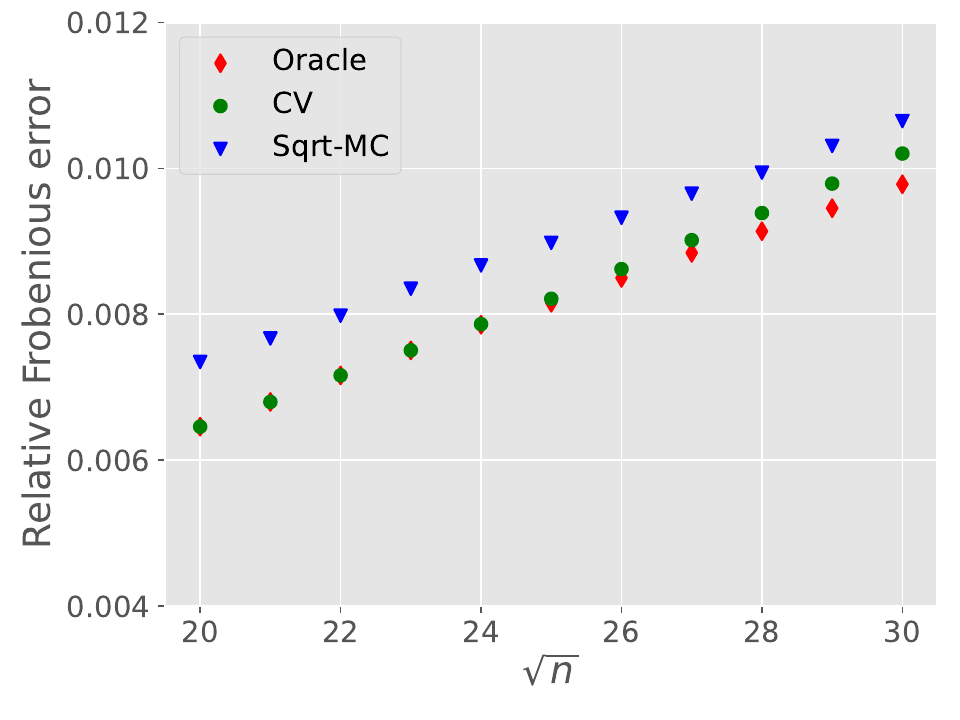}
\par\end{center}
\begin{center}
(a)
\par\end{center}%
\end{minipage}\hfill{}%
\begin{minipage}[t]{0.45\columnwidth}%
\begin{center}
\includegraphics[width=1\textwidth]{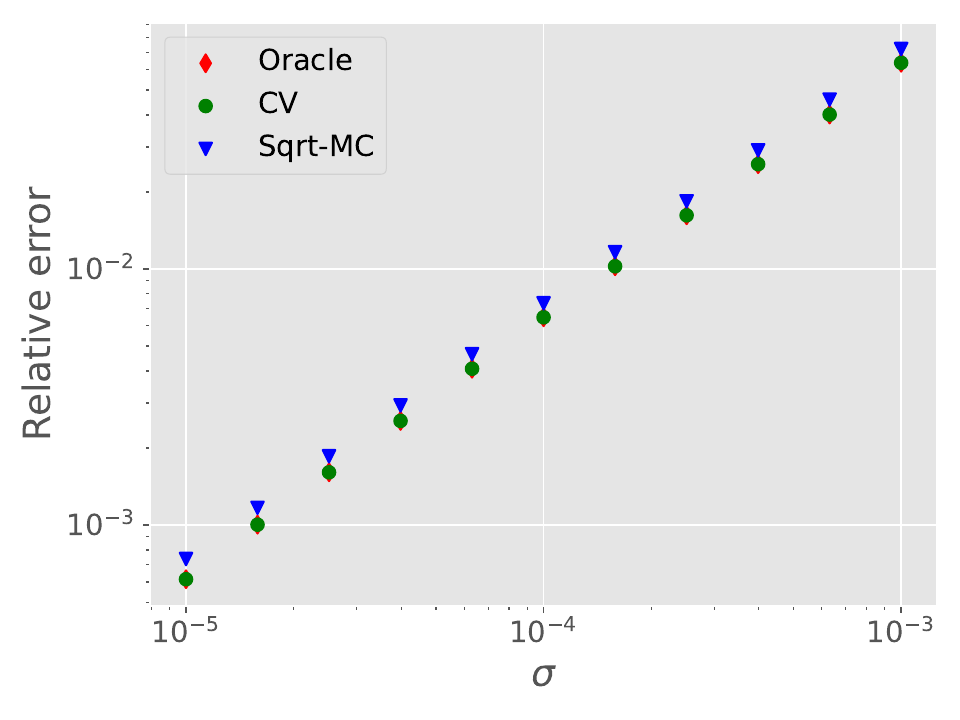}
\par\end{center}
\begin{center}
(b)
\par\end{center}%
\end{minipage}
\caption{\label{fig:oracle} (a) Relative Frobenius estimation error of 
\est~and solution of \eqref{eq:mc} with oracle and cross-validated $\lambda$ vs. problem size $\sqrt{n}$. The parameters are
fixed as $\sigma=10^{-4}, r=5, p=0.5$. (b) Relative Frobenius estimation error of \est~and solution of \eqref{eq:mc} with oracle and cross-validated $\lambda$ vs. noise size $\sigma$ on a log-log scale. The parameters are fixed as $n=400, r=5, p=0.5$. In both settings, $k = 10$ for the number of folds in cross validation and each point represents the average of 10 independent trials.}
\end{figure}

In each run of the experiment, we first generate an $n\times n$ matrix $\bm{M}$
as in \eqref{eq:obs} and calculate its estimator using \est~with fixed regularization parameter $\lambda = 2/\sqrt{n}$ and \eqref{eq:mc} with $\hat{\lambda}_{\mathrm{oracle}}$
and 10-fold cross-validated $\hat{\lambda}_{\mathrm{CV}}$. Figure~\ref{fig:oracle} shows the relative Frobenius
errors of the different methods across varying matrix size $n$ and varying noise level $\sigma$. In both settings, we can see while \est~has very close estimation error to that of \eqref{eq:mc}. Moreover their linear
trends over problem size $\sqrt{n}$ and noise size $\sigma$ are similar, as we expect from their identical error rate. This shows
that by using \est, we achieve the tuning-free property with a minor sacrifice in the rate of estimation performance.

\paragraph*{Performance with approximately low-rank matrices.}

Another point of interest is whether \est~is robust to misspecification of the low-rank assumption. Here we conduct the experiment with approximately
rank-$r$ matrices $\bm{L}^{\star}$ that singular values $\sigma_{1},\cdots,\sigma_{r}=1$
and $\sigma_{l}\propto(n-l)^{-2}$ such that $\sum_{l=r+1}^{n}\sigma_{l}\eqqcolon\gamma$.
This parameter $\gamma$ can be viewed as a measurement of deviation from the set rank-$r$ matrices, as
\[
\gamma=\min_{\bm{L}:\mathrm{rank}(\bm{L})=r}\|\bm{L}^{\star}-\bm{L}\|_{\ast}.
\]
We then perform the same experiments as above, i.e., comparing \est~to \eqref{eq:mc} with oracle and cross-validated $\lambda$. Figure \ref{fig:approx_lr_vs_oracle} shows their respective estimation error vs $\gamma$. We can see that the estimation error for all three methods increases when $\gamma$ increases and the increments are small and comparable across the three methods. This shows that \est~and \eqref{eq:mc} to are somewhat robust to the violation of low rank assumption.

In addition, we showcase an interesting discovery which compares the robustness of convex and nonconvex version of \est~to approximate low-rankness. We generate the ground-truth matrices that is approximately low rank and calculate \est~and the nonconvex solution of \eqref{eq:nonconvex-opt-theta} assuming the rank is $r$. Figure~\ref{fig:approx_lr} shows
that the performance of \est~for approximately low rank matrices
is close to the case with exact low-rankness ($\gamma=0$), while the nonconvex method suffers a much greater loss in estimation accuracy. The difference between convex and nonconvex method is close to 0 when $\gamma = 0$ and increases drastically as $\gamma$ increases. To some extent, this is expected as the convex method does not require the input of rank information.

\begin{figure}
\centering{}\includegraphics[width=0.45\textwidth]{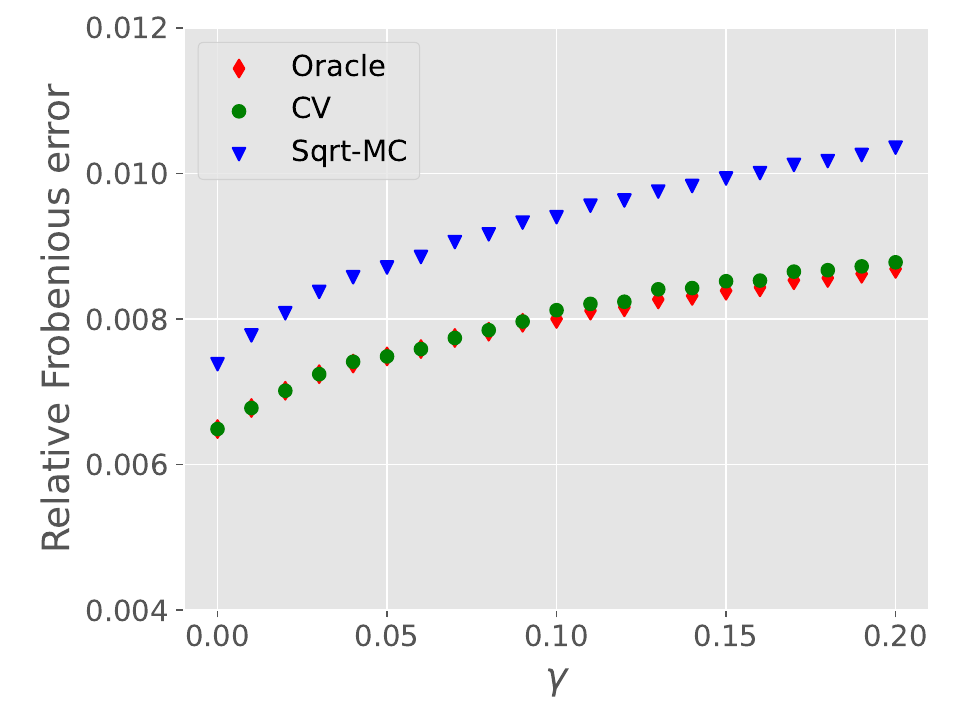}\caption{\label{fig:approx_lr_vs_oracle}  Relative Frobenius estimation error of 
\est~and solution of \eqref{eq:mc} with oracle and cross-validated $\lambda$ vs. $\gamma$ for approximately low rank matrices. The parameters are
chosen as: $n = 400,r=5,p=0.5, \sigma = 10^{-4}, \lambda = 2/\sqrt{n}$ while $\gamma$ varies from $0$ to $0.2$.
Each point represents the average of 10 independent trials.}
\end{figure}

\begin{figure}
\centering{}\includegraphics[width=0.45\textwidth]{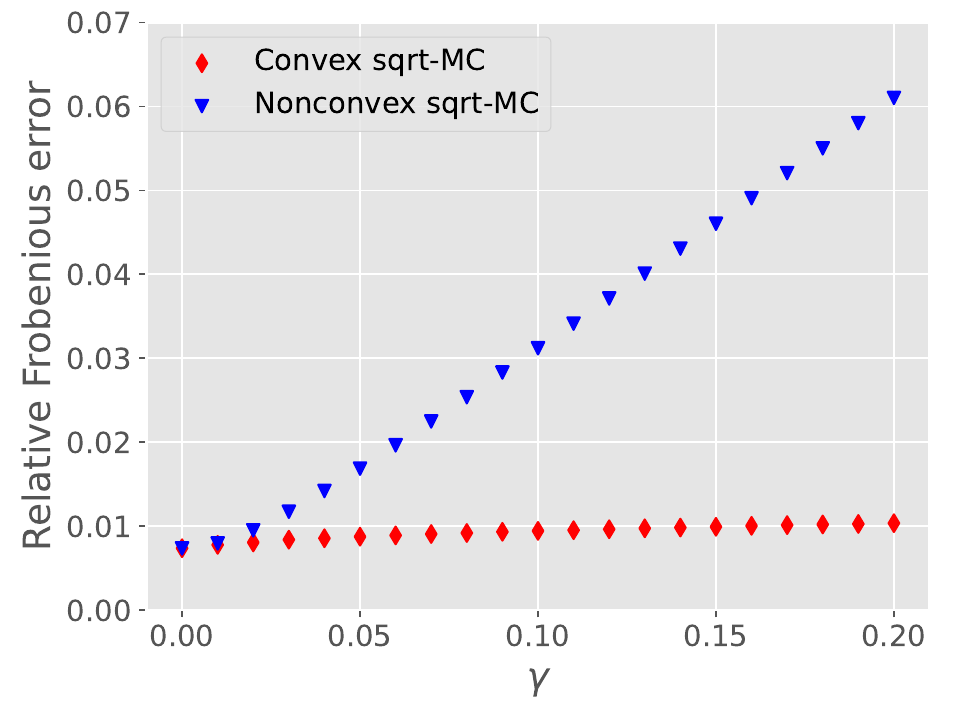}\caption{\label{fig:approx_lr}Relative Frobenius estimation error of 
\est~and \eqref{eq:nonconvex-opt-theta} for approximately low rank matrices. The parameters are
chosen as: $n = 400,r=5,p=0.5, \sigma=10^{-4}, \lambda = 2/\sqrt{n}$ while $\gamma$ varies from $0$ to $0.2$.
Each point represents the average of 10 independent trials.}
\end{figure}

\section{Prior art}

\label{sec:Prior-art}

\paragraph{Matrix completion.}  
Convex relaxation has been extensively studied for the matrix completion problem both in the noiseless setting~\cite{candes2009exact, candes2010power, gross2011recovering, recht2011simpler, chen2015incoherence}, and the noisy case~\cite{candes2010matrix, negahban2012restricted, koltchinskii2011nuclear, klopp2014noisy, chen2020noisy}. In the noiseless setting, convex relaxation achieves exact recovery as soon as the number of observed entries $n^2 p$ exceeds $n r \log n \log r$ \cite{ding2020leave}---roughly the degrees of freedom of a rank-$r$ matrix, which is information-theoretically optimal.  
When it comes to the noisy setting, Cand\`es and Plan~\cite{candes2010matrix} focuses on arbitrary noise (e.g., noise could be deterministic and adversarial), and proves that convex relaxation is stable w.r.t.~the noise size. The theoretical guarantees for convex relaxation are strengthened by Chen et al.~\cite{chen2020noisy} in the stochastic noise case, which is the same setting we study in the current paper. Such a discrepancy between stochastic and deterministic noise for convex relaxation is also documented in~\cite{krahmer2021convex}. 

Pioneered by the work~\cite{keshavan2010matrix, keshavan2009matrix}, nonconvex optimization has gained a lot of attentions during the past decade for solving matrix completion owing to its computational efficiency. Efficient computational and statistical guarantees have been provided for manifold optimization~\cite{keshavan2010matrix, keshavan2009matrix}, gradient descent~\cite{ma2018implicit, chen2020nonconvex}, projected gradient descent~\cite{chen2015fast, zheng2016convergence}, alternating minimization~\cite{jain2013low, hardt2014understanding}, scaled gradient descent~\cite{tong2021accelerating}, singular value projection~\cite{ding2020leave}, etc. See the recent surveys~\cite{davenport2016overview, chi2019nonconvex} for more related work on matrix completion.

% Matrix Completion has attracted a lot of attentions as it has appealing
% potentials in applications. It has extensive use (\cite{nguyen2019low,ramlatchan2018survey,li2019survey})
% in recommmendation system, signal processing, etc. Various algorithms
% have emerged. To name a few, there are Singular Value Thresholding
% \cite{cai2010singular}, variants of gradient descent \cite{Chen2020,ma2018implicit,chen2015fast,jin2016provable,sun2016guaranteed},
% and alternative minimization \cite{jain2013low,hardt2014understanding}.

% One major line of theoretical works relevant to our paper is in providing
% theoretical guarantees for the convex problem of nuclear norm minimization.
% \cite{candes2009exact} showed an exact recovery for matrix completion
% in a noiseless setting. Later on people made more progress in the
% noisy case. \cite{candes2010matrix} gave a bound that is $O(\sigma n^{3/2})$
% in our notations. \cite{negahban2012restricted} established
% the minimax lower round of the estimation error and presented an upper
% bound that is optimal for $\sigma>\|\bm{L}^{\star}\|_{\infty}$. More
% recently \cite{chen2020noisy} proved that nuclear norm minimization
% achieves the minimax optimal rate of $O(\sigma\sqrt{n/p})$ via an
% analysis of the nonconvex solution similar to our approach in this
% work. However, most of the works in this direction rely on a choice
% of regularization parameter $\lambda$ that is dependent on the noise
% level $\sigma$.

\paragraph{Tuning-free methods.}
A variety of tuning-free methods have been proposed to tackle high-dimensional linear regression. The seminal work~\cite{belloni2011square}
proposes the square-root Lasso estimator which does not rely on knowing the size of the noise and is also statistically optimal. \cite{sun2012scaled} proposes an equivalent method named scaled sparse linear regression, which originates from the concomitant scale estimation~\cite{huber2011robust, owen2007robust}. \cite{lederer2015don} proposes TREX, a method similar to square-root Lasso and
is completely parameter-free. \cite{wang2020tuning} borrows ideas from non-parametric statistics and proposes Rank Lasso, whose optimal choice of tuning parameter can be simulated easily in the
case with unknown variance of the noise. See~\cite{wu2019survey}
for a survey on the selection of tuning-parameters for high-dimensional
regression and \cite{giraud2012high} for a survey on regression with
unknown variance of noise.

% Now we move to the matrix estimation problem, which is the focus of our paper. 
% Inspired by square-root Lasso, \cite{klopp2014noisy} proposes \est~with an extra max-norm constraint. 
% More recently, \cite{zhang2021square} extends the idea to the robust PCA problem. 

% In addition, \cite{gaiffas2017high} proposed a $\sigma$-independent
% choice of regularization parameter for the penalized $\ell_{2}$ loss
% optimization \eqref{eq:mc} that achieves the same sub-optimal rate
% as \cite{negahban2012restricted,klopp2014noisy}. \cite{gilbert2019tuning}
% looked at the Bayesian approach and proposed a method without hyper-parameters
% using global-local priors like horseshoe prior.

\paragraph{Bridging convex and nonconvex optimization.} The connections between convex and
nonconvex optimization has been extensively used in a recent line of work. Chen et al.~\cite{chen2020noisy} uses this to prove the optimality of the vanilla least-squares estimator for noisy matrix completion; Later, the papers~\cite{chen2021bridging, chen2021convex, wang2022robust} extend the technique to the robust PCA problem, the blind deconvolution problem, and matrix completion with heavy-tailed noise. 

\paragraph{Leave-one-out analysis.} Leave-one-out analysis is powerful in decoupling statistical dependence and obtain element-wise performance guarantees. It has been successfully applied to high-dimensional regression~\cite{el2013robust, el2018impact}, phase synchronization~\cite{zhong2018near}, ranking~\cite{chen2019spectral, chen2022partial}, matrix completion~\cite{ma2018implicit, chen2020nonconvex, abbe2020entrywise}, reinforcement learning~\cite{pananjady2020instance}, high-dimensional inference~\cite{chen2019inference, yan2021inference} to name a few. Interested readers are referred to a recent overview~\cite{chen2021spectral} for detailed discussions.  

\section{Discussions}

Focusing on the noisy matrix completion problem, this paper shows
that a tuning-free estimator---\est~achieves
optimal statistical performance. This opens up several interesting
avenues for future research. Below, we list a few of them. 
\begin{itemize}
\item \emph{Extensions to robust PCA. }While our work focuses on matrix
completion, a natural extension is to further consider partial observations
with outliers, i.e., robust PCA. As mentioned, Zhang et al.~\cite{zhang2021square}
has studied this problem (with full observation) and provides an error
guarantee of order $O(\sigma n^{2})$, which is sub-optimal in its
dependency on the problem dimension. By contrast, a vanilla least-squares
estimator with noise-size-dependent choice of $\lambda$ has been
shown to be optimal~\cite{chen2021bridging}. It remains to be seen
whether one can devise an optimal tuning-free method for robust PCA
with noise and missing data. 
\item \emph{Inference for }\est~\emph{estimator.} The current paper discusses
solely the estimation performance of the tuning-free estimator. As
statistical inference for matrix completion is equally important,
one wishes to develop inferential procedures around the \est~estimator
as that has been done in the paper~\cite{chen2019inference} for
the vanilla least-squares estimator. 

\item \emph{Robustness to non-uniform design.} In high-dimensional linear regression, optimal tuning-free methods have been developed to be adaptive to both the unknown noise size and the design matrix. In the matrix completion setting, the design is governed by the sampling pattern, which is assumed to be uniform in the current paper. It is of great interest to develop robust and tuning-free approaches for noisy matrix completion with non-uniform sampling that improve over the max-norm constrained estimator in~\cite{klopp2014noisy}. 
\end{itemize}

\bibliographystyle{alpha}
\bibliography{All-of-Bibs}

\appendix

\section{Proof of Lemma~\ref{lemma:induction}} \label{sec:proof-induction}

We prove Lemma~\ref{lemma:induction} via induction.
Since all the algorithms start from the groundtruth, it is trivial
to see that the hypotheses~\eqref{eq:ncvx-iterates-full} hold for
$t=0$. We also record two important properties of the iterates at
$t=0$, namely,   
\begin{equation}
\frac{1}{2}np^{1/2}\sigma\le\theta_{t}\le2np^{1/2}\sigma\label{eq:ncvx-itr-bound-theta}
\end{equation}
and 
\begin{equation}
\|\bm{X}_{t}^{\top}\bm{X}_{t}-\bm{Y}_{t}^{\top}\bm{Y}_{t}\|_{\mathrm{F}}\le C_{\mathrm{B}}\kappa\eta\frac{\sigma}{\sigma_{\min}}\sqrt{\frac{n}{p}}\sqrt{r}\sigma_{\max}^{2},\label{eq:ncvx-itr-bound-balance}
\end{equation}
where $C_{\mathrm{B}} > 0$ is a universal constant. Note that at $t=0$, we have $\theta_0 = \|\mathcal{P}_\Omega(\bm{E})\|_{\mathrm{F}}$, which concentrates sharply around $np^{1/2}\sigma$ under the noise assumption and uniform sampling. 

Now suppose the hypotheses~\eqref{eq:ncvx-iterates-full}, \eqref{eq:ncvx-itr-bound-theta},
and \eqref{eq:ncvx-itr-bound-balance} hold for the $t$-the iterates.
We aim to show that the same set of hypotheses continue to hold for
the $(t+1)$-th iterates. Sections~\ref{sec:induction-1} and~\ref{sec:induction-2} are devoted to this induction step. In addition, we prove the last claim~\eqref{eq:small-grad} in Section~\ref{sec:gradient-size}. In Section~\ref{subsec:proof-lemma-error-size} we prove Lemma~\ref{lemma:error-size} which is a consequence of \eqref{eq:ncvx-iterates-full} and \eqref{eq:ncvx-itr-bound-theta}.

\subsection{Induction on hypotheses~\eqref{eq:ncvx-iterates-full} and \eqref{eq:ncvx-itr-bound-balance}} \label{sec:induction-1}

Define 
\[
\tilde{\lambda}_{t}\coloneqq\lambda\theta_{t},\qquad\text{and}\qquad\tilde{\eta}_{t}\coloneqq\eta/\theta_{t}.
\]
We make a key observation that the $t$-th iterations of Algorithm~\ref{algo:non-cvx-gd}
and \ref{algo:leave-one-out} are exactly the same as the $t$-th iterations
of Algorithm 1 (vanilla gradient descent) and 2 (construction of the
leave-one-out sequence) in the paper \cite{chen2020noisy} with the
parameters $\tilde{\lambda}_{t}$ and $\tilde{\eta}_{t}$. Moreover,
given the induction hypothesis \eqref{eq:ncvx-itr-bound-theta} one
has $\frac{1}{2}n\sqrt{p}\sigma\le\theta_{t-1}\le2n\sqrt{p}\sigma$.
Combine this with our choice of $\lambda=C_{\lambda}n^{-1/2}$ to
see that 
\[\tilde{\lambda}_t\asymp \sigma\sqrt{np}
,\qquad\text{and\ensuremath{\qquad}\ensuremath{\tilde{\eta}_{t}\asymp1/(np\kappa^{3}\sigma_{\max})}}.
\]which are consistent with the choice of $\lambda$ and $\eta$  in \cite{chen2020noisy}. These allow us to invoke Lemmas 10-15 in \cite{chen2020noisy} to
prove that claims \eqref{eq:ncvx-iterates-full} and \eqref{eq:ncvx-itr-bound-balance}
hold for the $(t+1)$-th iterates. 

\subsection{Induction on hypotheses~\eqref{eq:ncvx-itr-bound-theta}} \label{sec:induction-2}

In this section, we aim to show that the claim \eqref{eq:ncvx-itr-bound-theta}
holds for the $(t+1)$-th iterates. 

Observe that
\[
\mathcal{P}_{\Omega}\left(\bm{X}_{t+1}\bm{Y}_{t+1}^{\top}-\bm{M}\right)=\mathcal{P}_{\Omega}\left(\bm{X}_{t+1}\bm{Y}_{t+1}^{\top}-\bm{L}^{\star}\right)-\mathcal{P}_{\Omega}(\bm{E})\text{.}
\]
Similar to the proof of Lemma~\ref{lemma:ncvx-iterates-bound-XY-L},
using the incoherence assumption $\|\bm{F}^{\star}\|_{2,\infty}=\max\left\{ \|\bm{X}^{\star}\|_{2,\infty},\|\bm{Y}^{\star}\|_{2,\infty}\right\} \le\sqrt{\mu r\sigma_{\max}/n},$we
have 
\begin{align*}
\|\bm{X}_{t+1}\bm{Y}_{t+1}^{\top}-\bm{L}^{\star}\|_{\infty} & \le3C_{\mathrm{\infty}}\frac{\sigma}{\sigma_{\min}}\sqrt{\frac{n\log n}{p}}\|\bm{F}^{\star}\|_{2,\infty}\|\bm{F}^{\star}\|_{2,\infty}\\
 & \le3C_{\mathrm{\infty}}\frac{\mu r\sigma_{\max}}{n}\frac{\sigma}{\sigma_{\min}}\sqrt{\frac{n\log n}{p}}.
\end{align*}
Then 
$
\left\Vert \mathcal{P}_{\Omega}\left(\bm{X}_{t+1}\bm{Y}_{t+1}^{\top}-\bm{L}^{\star}\right)\right\Vert _{\mathrm{F}}  \lesssim n\sqrt{p}\|\bm{X}_{t+1}\bm{Y}_{t+1}^{\top}-\bm{L}^{\star}\|_{\infty}
  \lesssim\kappa\mu r\sigma\sqrt{n\log n}.
$
As the sample size satisfies $n^{2}p\gg\kappa^{4}\mu^{2}r^{2}n\log^{3}n$,
we have $\|\mathcal{P}_{\Omega}\left(\bm{X}_{t+1}\bm{Y}_{t+1}^{\top}-\bm{L}^{\star}\right)\|_{\mathrm{F}}\ll n\sqrt{p}\sigma$.
As mentioned before, $\| \mathcal{P}_\Omega(\bm{E}) \|_{\mathrm{F}}$ sharply concentrates around $np^{1/2}\sigma$. Therefore by the triangle inequality, we have
\begin{align*}
\frac{1}{2}\sigma n\sqrt{p}\le\|\mathcal{P}_{\Omega}\left(\bm{X}_{t}\bm{Y}_{t}^{\top}-\bm{M}\right)\|_{\mathrm{F}} & \le2\sigma n\sqrt{p}
\end{align*}
for large enough $n$.

\subsection{Proof of bound~\eqref{eq:small-grad}} \label{sec:gradient-size}

Suppose
for the moment that 
\begin{equation}
f(\boldsymbol{X}_{t},\boldsymbol{Y}_{t},\theta_{t})\le f(\boldsymbol{X}_{t-1},\boldsymbol{Y}_{t-1},\theta_{t-1})-\frac{\eta}{2}\|\nabla_{\bm{X},\bm{Y}}f(\boldsymbol{X}_{t-1},\boldsymbol{Y}_{t-1},\theta_{t-1})\|_{\mathrm{F}}^{2}\label{eq:obj-decrease}
\end{equation}
holds for all $t\geq1$. Then a telescoping argument would yield the
conclusion that 
\begin{align*}
f(\bm{X}_{0},\bm{Y}_{0},\theta_{0})-f(\bm{X}_{t_{0}},\bm{Y}_{t_{0}},\theta_{t_{0}}) & \ge\frac{\eta}{2}\sum_{t=0}^{t_{0}-1}\|\nabla_{\bm{X},\bm{Y}}f(\bm{X}_{t},\bm{Y}_{t},\theta_{t})\|_{\mathrm{F}}^{2}\\
 & \ge\frac{\eta t_{0}}{2}\min_{0\le t<t_{0}}\|\nabla_{\bm{X},\bm{Y}}f(\bm{X}_{t},\bm{Y}_{t},\theta_{t})\|_{\mathrm{F}}^{2}.
\end{align*}
Expanding the left hand side, we see that it is upper bounded by 
\begin{align*}
f(\bm{X}_{0},\bm{Y}_{0},\theta_{0})-f(\bm{X}_{t_{0}},\bm{Y}_{t_{0}},\theta_{t_{0}}) & =\|\mathcal{P}_{\Omega}\text{(}\bm{E})\|_{\mathrm{F}}-\|\mathcal{P}_{\Omega}\text{(}\bm{X}_{t_{0}}\bm{Y}_{t_{0}}^{\top}-\bm{M})\|_{\mathrm{F}}\\
&\quad+\frac{\lambda}{2}\left(\|\bm{X}^{\star}\|_{\mathrm{F}}^{2}-\|\bm{X}_{t_{0}}\|_{\mathrm{F}}^{2}+\|\bm{Y}^{\star}\|_{\mathrm{F}}^{2}-\|\bm{Y}_{t_{0}}\|_{\mathrm{F}}^{2}\right)\\
 & \le\|\mathcal{P}_{\Omega}\text{(}\bm{E})\|_{\mathrm{F}}+\frac{\lambda}{2}\left(\|\bm{X}^{\star}\|_{\mathrm{F}}^{2}-\|\bm{X}_{t_{0}}\bm{H}_{t_{0}}\|_{\mathrm{F}}^{2}+\|\bm{Y}^{\star}\|_{\mathrm{F}}^{2}-\|\bm{Y}_{t_{0}}\bm{H}_{t_{0}}\|_{\mathrm{F}}^{2}\right),
\end{align*}
where the last line uses the nonnegativity of norms and the invariance
of Frobenius norm under rotation. In view of the properties~\eqref{eq:ncvx-iterates-full}
and the noise size assumption $\frac{\sigma}{\sigma_{\min}}\sqrt{\frac{n}{p}}\ll1$, 
we have 
\[
\|\bm{X}^{\star}-\bm{X}_{t_{0}}\bm{H}_{t_{0}}\|_{\mathrm{F}}\lesssim\frac{\sigma}{\sigma_{\min}}\sqrt{\frac{n}{p}}\|\bm{X}^{\star}\|_{\mathrm{F}},\qquad\text{and\ensuremath{\qquad}\ensuremath{\|\bm{X}_{t_{0}}\|}}_{\mathrm{F}}=\|\bm{X}_{t_{0}}\bm{H}_{t_{0}}\|_{\mathrm{F}}\le2\|\bm{X}^{\star}\|_{\mathrm{F}}.
\]
Then, 
\begin{align}
\Bigl|\|\bm{X}^{\star}\|_{\mathrm{F}}^{2}-\|\bm{X}_{t_{0}}\bm{H}_{t_{0}}\|_{\mathrm{F}}^{2}\Bigr| & \le\Bigl|\|\bm{X}^{\star}\|_{\mathrm{F}}-\|\bm{X}_{t_{0}}\bm{H}_{t_{0}}\|_{\mathrm{F}}\Bigr|\left(\|\bm{X}^{\star}\|_{\mathrm{F}}+\|\bm{X}_{t_{0}}\bm{H}_{t_{0}}\|_{\mathrm{F}}\right)\label{eq:X-fro-maximal-change}\\
 & \lesssim\frac{\sigma}{\sigma_{\min}}\sqrt{\frac{n}{p}}\|\bm{X}^{\star}\|_{\mathrm{F}}\|\bm{X}^{\star}\|_{\mathrm{F}}\nonumber \\
 & \le\sigma r\kappa\sqrt{\frac{n}{p}},\nonumber 
\end{align}
where the last line uses the fact that $\|\bm{X}^{\star}\|_{\mathrm{F}}\le\sqrt{r\sigma_{\max}}$.
Similarly, we have  
$
\left|\|\bm{Y}^{\star}\|_{\mathrm{F}}^{2}-\|\bm{Y}_{t_{0}}\bm{H}_{t_{0}}\|_{\mathrm{F}}^{2}\right|\lesssim\sigma r\kappa\sqrt{\frac{n}{p}}$.
These combined with the fact that $\|\mathcal{P}_{\Omega}\text{(}\bm{E})\|_{\mathrm{F}}\lesssim n\sqrt{p}\sigma$ implies ,
as $t_{0}=n^{18},\eta\asymp\sigma/(\sqrt{p}\kappa^{3}\sigma_{\max})$,
and $\lambda\asymp1/\sqrt{n}$, 
\begin{align*}
\mathbf{}\min_{0\le t<t_{0}}\|\nabla_{\bm{X},\bm{Y}}f(\bm{X}_{t},\bm{Y}_{t},\theta_{t})\|_{\mathrm{F}} & \le\left[\frac{f(\bm{X}_{0},\bm{Y}_{0},\theta_{0})-f(\bm{X}_{t_{0}},\bm{Y}_{t_{0}},\theta_{t_{0}})}{\eta t_{0}/2}\right]^{1/2}\\
 & \lesssim\left[\frac{\sigma_{\max}}{n^{18}\sqrt{p}\sigma}\left(\frac{n\sigma}{\sqrt{p}}\right)\right]^{1/2}\\
 & \lesssim\frac{1}{n^{8}}\sqrt{\frac{\sigma_{\max}}{p}}.
\end{align*}
To simplify the expression we use $\kappa\lesssim n$ and $r\lesssim\sqrt{n}$
which are consequences of the sample size assumption $n^{2}\ge n^{2}p\gg\kappa^{4}\mu^{2}r^{2}n\log n$.

\paragraph{Proof of bound~\eqref{eq:obj-decrease}. }

Define $h(\bm{X},\bm{Y})\coloneqq\theta_{t}\left[f(\bm{X},\bm{Y},\theta_{t})-\theta_{t}/2\right]$.
Then $h(\bm{X},\bm{Y})$ matches the form of the objective function
in Lemma 16 of the paper~\cite{chen2020noisy}. Then Lemma 16 therein tells us that 
\[
h(\bm{X}_{t+1},\bm{Y}_{t+1})\le h(\bm{X}_{t},\bm{Y}_{t})-\frac{\tilde{\eta}_{t}}{2}\|\nabla h(\bm{X}_{t},\bm{Y}_{t})\|_{\mathrm{F}}^{2},
\]
where we recall $\tilde{\eta}_{t}=\eta/\theta_{t}$. Rewriting the bound in terms of $f$ yields
\begin{equation}
f(\bm{X}_{t+1},\bm{Y}_{t+1},\theta_{t})\le f(\bm{X}_{t},\bm{Y}_{t},\theta_{t})-\frac{\eta}{2}\|\nabla_{\bm{X},\bm{Y}}f(\bm{X}_{t},\bm{Y}_{t},\theta_{t})\|_{\mathrm{F}}^{2}.\label{eq:f-decrease-f-grad}
\end{equation}
In addition, by the optimality of $\theta_{t+1}$, one has
\begin{equation}
f(\bm{X}_{t+1},\bm{Y}_{t+1},\theta_{t+1})\le f(\bm{X}_{t+1},\bm{Y}_{t+1},\theta_{t}).\label{eq:f-decrease-theta}
\end{equation}
Combining equations~\eqref{eq:f-decrease-f-grad} and~\eqref{eq:f-decrease-theta}
completes the proof.

\subsection{Proof of Lemma~\ref{lemma:error-size}\label{subsec:proof-lemma-error-size}}

By Lemma~\ref{lemma:induction}, we know that $\bm{X}_{\mathrm{ncvx}}$
satisfies
\begin{align*}
\|\bm{X}_{\mathrm{ncvx}}-\bm{X}^{\star}\| & \le C_{\mathrm{op}}\left(\frac{\sigma}{\sigma_{\min}}\sqrt{\frac{n}{p}}\right)\|\bm{X}^{\star}\| \ll \sqrt{\sigma_{\min}},
\end{align*}
where the last relation arises from the noise level assumption $\frac{\sigma}{\sigma_{\min}}\sqrt{\frac{n}{p}}\text{\ensuremath{\ll1/\sqrt{\kappa^{4}\mu r\log n}}}$. Therefore we can apply Weyl's inequality to obtain 
\begin{align*}
\sigma_{\max}(\bm{X}_{\mathrm{ncvx}}) & \le\sqrt{\sigma_{\max}}+\|\bm{X}_{\mathrm{ncvx}}-\bm{X}^{\star}\|\le\sqrt{2\sigma_{\max}};\\
\sigma_{\min}(\bm{X}_{\mathrm{ncvx}}) & \ge\sqrt{\sigma_{\min}}-\|\bm{X}_{\mathrm{ncvx}}-\bm{X}^{\star}\|\ge\sqrt{\sigma_{\min}/2}
\end{align*}
for large enough $n$. These hold similarly for the singular values of
$\bm{Y}_{\mathrm{ncvx}}$. 

On the other hand, the relations \eqref{eq:theta-bound-Lncvx} come
directly from \eqref{eq:ncvx-itr-bound-theta}, and \eqref{eq:regularization-req-P}
follows from Lemma~4 in \cite{chen2020noisy}. 

\section{Proof of Lemma \ref{lemma:ncvx-to-cvx-full}}

\label{subsec:Proof-of-Lemma-ncvx-cvx-full}

To simplify the notation, we denote $\theta\coloneqq\mathcal{\|P}_{\Omega}(\bm{L}_{\mathrm{ncvx}}-\bm{M})\|_{\mathrm{F}}$,
and $\bm{\Delta}\coloneqq\bm{L}_{\mathrm{cvx}}-\bm{L}_{\mathrm{ncvx}}$
throughout this section. In view of Lemma~\ref{lemma:error-size},
we know that $\theta\neq0$, and hence $\theta^{-1}$ is well defined. 

Recall that $\bm{U}\bm{\Sigma}\bm{V}^{\top}$ is the SVD for $\bm{L}_{\mathrm{ncvx}}$,
and $T$ is the tangent space at $\bm{L}_{\mathrm{ncvx}}$. The following
lemma is useful in controlling the size of $\bm{\Delta}$. 

\begin{lemma}\label{lemma:residual-size-in-T}

Under the notations and assumptions of Lemma~\ref{lemma:ncvx-to-cvx-full},
we have
\begin{equation}
\tfrac{1}{\theta}\mathcal{P}_{\Omega}(\bm{L}_{\mathrm{ncvx}}-\bm{M})=-\lambda(\boldsymbol{U}\boldsymbol{V}^{\top}+\bm{R}),\label{eq:f-gradient-residual}
\end{equation}
where $\bm{R}$ is a residual matrix such that 
\begin{align*}
\|\mathcal{P}_{T}(\bm{R})\|_{\mathrm{F}} & \le70\kappa\sigma_{\min}^{-1/2}\|\nabla g(\bm{X},\bm{Y})\|_{\mathrm{F}},\qquad\text{and}\qquad\|\mathcal{P}_{T^{\perp}}(\bm{R})\|<1/2.
\end{align*}

\end{lemma}

\noindent  See Section \ref{subsec:Proof-Gradient-Decomp} for the
proof. 

\medskip

We decompose the proof into three steps. In Step 1, we show that the
difference matrix $\bm{\Delta}$ mainly lies in the tangent space
$T$. In Step 2, the previous fact is leveraged to show an upper bound
on $\mathcal{P}_{\Omega}(\bm{\Delta})$. In the last step (Step 3),
we connect the previous steps with the injectivity property (cf.~Lemma~\ref{lemma:injectivity})
to reach the desired conclusion. 

\paragraph{Step 1: showing that $\bm{\Delta}$ lies primarily in the tangent
space $T$.}

By the optimality of $\bm{L}_{\mathrm{cvx}}$, we have 
\begin{equation}
0\ge\|\mathcal{P}_{\Omega}(\bm{L}_{\mathrm{cvx}}-\bm{M})\|_{\mathrm{F}}-\|\mathcal{P}_{\Omega}(\bm{L}_{\mathrm{ncvx}}-\bm{M})\|_{\mathrm{F}}+\lambda\left(\|\bm{L}_{\mathrm{cvx}}\|_{*}-\|\bm{L}_{\mathrm{ncvx}}\|_{*}\right).\label{eq:optimality-Lcvx}
\end{equation}
 Use the convexity of $\|\cdot\|_{\mathrm{F}}$ and $\|\cdot\|_{*}$
and the decomposition $\bm{L}_{\mathrm{ncvx}}=\bm{U}\bm{\Sigma}\bm{V}^{\top}$
to see that 
\begin{align*}
0 & \ge\left\langle \frac{1}{\theta}\mathcal{P}(\bm{L}_{\mathrm{ncvx}}-\bm{M}),\bm{\Delta}\right\rangle +\lambda\left\langle \boldsymbol{U}\boldsymbol{V}^{\top}+\bm{W}_{0},\bm{\Delta}\right\rangle
\end{align*}
holds for any $\bm{W}_{0}\in T^{\perp}$ with $\|\bm{W}_{0}\|\leq1$.
Apply Lemma~\ref{lemma:residual-size-in-T} to further obtain
\[
0\geq-\lambda\left\langle \bm{R},\bm{\Delta}\right\rangle +\lambda\left\langle \bm{W}_{0},\bm{\Delta}\right\rangle .
\]
In particular, one can choose $\bm{W}_{0}\in T^{\perp}$ such that
$\|\mathcal{P}_{T^{\perp}}(\boldsymbol{\Delta})\|_{*}=\langle\bm{W}_{0},\bm{\Delta}\rangle$,
which yields the inequality 
\begin{align*}
0 & \geq\lambda\|\mathcal{P}_{T^{\perp}}(\boldsymbol{\Delta})\|_{*}-\lambda\left\langle \bm{R},\bm{\Delta}\right\rangle \\
 & =\lambda\|\mathcal{P}_{T^{\perp}}(\boldsymbol{\Delta})\|_{*}-\lambda\left\langle \mathcal{P}_{T}(\bm{R}),\bm{\Delta}\right\rangle -\lambda\left\langle \mathcal{P}_{T^{\perp}}(\bm{R}),\bm{\Delta}\right\rangle \\
 & \geq\lambda\|\mathcal{P}_{T^{\perp}}(\boldsymbol{\Delta})\|_{*}-\lambda\|\mathcal{P}_{T}(\bm{R})\|_{\mathrm{F}}\|\mathcal{P}_{T}(\boldsymbol{\Delta})\|_{\mathrm{F}}-\lambda\|\mathcal{P}_{T^{\perp}}(\bm{R})\|\|\mathcal{P}_{T^{\perp}}(\boldsymbol{\Delta})\|_{*}.
\end{align*}
Here the last line arises from Holder's inequality. 

Again, by Lemma~\ref{lemma:residual-size-in-T}, we have the bounds
$\|\mathcal{P}_{T}(\bm{R})\|_{\mathrm{F}}\le70\kappa\sigma_{\min}^{-1/2}\|\nabla g(\bm{X},\bm{Y})\|_{\mathrm{F}}$
and $\|\mathcal{P}_{T^{\perp}}(\bm{R})\|<1/2$, which allow us to
further arrive at 
\[
0\geq\frac{\lambda}{2}\|\mathcal{P}_{T^{\perp}}(\boldsymbol{\Delta})\|_{*}-70\lambda\kappa\sigma_{\min}^{-1/2}\|\nabla g(\bm{X},\bm{Y})\|_{\mathrm{F}}\|\mathcal{P}_{T}(\boldsymbol{\Delta})\|_{\mathrm{F}}.
\]
This further implies 

\begin{equation}
\|\mathcal{P}_{T^{\perp}}(\boldsymbol{\Delta})\|_{\mathrm{F}}\le\|\mathcal{P}_{T^{\perp}}(\boldsymbol{\Delta})\|_{*}\le140\kappa\sigma_{\min}^{-1/2}\|\nabla g(\bm{X},\bm{Y})\|_{\mathrm{F}}\|\mathcal{P}_{T}(\boldsymbol{\Delta})\|_{\mathrm{F}}.\label{eq:PTperp-Delta-PT-Delta}
\end{equation}
As an immediate consequence, under the assumed upper bound~\eqref{eq:gradient-size-bound}
for $\|\nabla g(\bm{X},\bm{Y})\|_{\mathrm{F}}$, we have \[{140\kappa\sigma_{\min}^{-1/2}\|\nabla g(\bm{X},\bm{Y})\|_{\mathrm{F}}\le1},\]
and hence 
\begin{equation}
\|\boldsymbol{\Delta}\|_{\mathrm{F}}\le\|\mathcal{P}_{T^{\perp}}(\boldsymbol{\Delta})\|_{\mathrm{F}}+\|\mathcal{P}_{T}(\boldsymbol{\Delta})\|_{\mathrm{F}}\le2\|\mathcal{P}_{T}(\boldsymbol{\Delta})\|_{\mathrm{F}}\label{eq:Delta-PTDelta}
\end{equation}

\paragraph{Step 2: bounding $\|\mathcal{P}_{\Omega}(\bm{\Delta})\|_{\mathrm{F}}^{2}$.}

We start with presenting an identity involving $\|\mathcal{P}_{\Omega}(\bm{\Delta})\|_{\mathrm{F}}^{2}$:
\begin{align}
\|\mathcal{P}_{\Omega}(\bm{\Delta})\|_{\mathrm{F}}^{2} & =\left(\|\mathcal{P}_{\Omega}(\bm{L}_{\mathrm{cvx}}-\boldsymbol{M})\|_{\mathrm{F}}-\|\mathcal{P}_{\Omega}(\bm{L}_{\mathrm{ncvx}}-\bm{M})\|_{\mathrm{F}}\right)\left(\|\mathcal{P}_{\Omega}(\bm{L}_{\mathrm{cvx}}-\boldsymbol{M})\|_{\mathrm{F}}+\|\mathcal{P}_{\Omega}(\bm{L}_{\mathrm{ncvx}}-\bm{M})\|_{\mathrm{F}}\right)\nonumber \\
 & \quad-2\langle\bm{\Delta},\mathcal{P}_{\Omega}(\bm{L}_{\mathrm{ncvx}}-\bm{M})\rangle\nonumber \\
 & =\left(\|\mathcal{P}_{\Omega}(\bm{L}_{\mathrm{cvx}}-\boldsymbol{M})\|_{\mathrm{F}}-\|\mathcal{P}_{\Omega}(\bm{L}_{\mathrm{ncvx}}-\bm{M})\|_{\mathrm{F}}\right)^{2}\label{eq:P-Omega-Delta-Square}\\
 & \quad+2\|\mathcal{P}_{\Omega}(\bm{L}_{\mathrm{ncvx}}-\bm{M})\|_{\mathrm{F}}\cdot\bigg(\|\mathcal{P}_{\Omega}(\bm{L}_{\mathrm{cvx}}-\boldsymbol{M})\|_{\mathrm{F}}-\|\mathcal{P}_{\Omega}(\bm{L}_{\mathrm{ncvx}}-\bm{M})\|_{\mathrm{F}}\nonumber\\
 &\qquad\qquad\qquad\qquad\qquad\qquad\qquad-\left\langle \frac{1}{\theta}\mathcal{P}_{\Omega}(\bm{L}_{\mathrm{ncvx}}-\bm{M}),\bm{\Delta}\right\rangle \bigg).\nonumber 
\end{align}
Lemma~\ref{lemma:residual-size-in-T} and Equation~\eqref{eq:optimality-Lcvx}
tell us that 
\begin{align*}
 & \|\mathcal{P}_{\Omega}(\bm{L}_{\mathrm{cvx}}-\boldsymbol{M})\|_{\mathrm{F}}-\|\mathcal{P}_{\Omega}(\bm{L}_{\mathrm{ncvx}}-\bm{M})\|_{\mathrm{F}}-\left\langle \frac{1}{\theta}\mathcal{P}_{\Omega}(\bm{L}_{\mathrm{ncvx}}-\bm{M}),\bm{\Delta}\right\rangle \\
 & \quad\leq\lambda\|\bm{L}_{\mathrm{ncvx}}\|_{*}-\lambda\|\bm{L}_{\mathrm{cvx}}\|_{*}+\lambda\left\langle \bm{U}\bm{V}^{\top}+\bm{R},\bm{\Delta}\right\rangle .
\end{align*}
By convexity of $\|\cdot\|_{*}$, this further simplifies to 
\begin{align}
 & \lambda\|\bm{L}_{\mathrm{ncvx}}\|_{*}-\lambda\|\bm{L}_{\mathrm{cvx}}\|_{*}+\lambda\left\langle \bm{U}\bm{V}^{\top}+\bm{R},\bm{\Delta}\right\rangle \nonumber \\
 & \quad\leq-\lambda\left\langle \bm{U}\bm{V}^{\top}+\bm{W},\bm{\Delta}\right\rangle +\lambda\left\langle \bm{U}\bm{V}^{\top}+\bm{R},\bm{\Delta}\right\rangle \nonumber \\
 & \quad=\lambda\langle\bm{\Delta},\bm{R}-\bm{W}\rangle,\label{eq:Delta-R-W}
\end{align}
for any $\bm{W}\in T^{\perp}$ with $\|\bm{W}\|\le1$. Combine Equation~\eqref{eq:P-Omega-Delta-Square}
and \eqref{eq:Delta-R-W} to reach 
\begin{align*}
\|\mathcal{P}_{\Omega}(\bm{\Delta})\|_{\mathrm{F}}^{2} & \leq\underbrace{\left(\|\mathcal{P}_{\Omega}(\bm{L}_{\mathrm{cvx}}-\boldsymbol{M})\|_{\mathrm{F}}-\|\mathcal{P}_{\Omega}(\bm{L}_{\mathrm{ncvx}}-\bm{M})\|_{\mathrm{F}}\right)^{2}}_{\eqqcolon\alpha_{1}}+\underbrace{2\lambda\|\mathcal{P}_{\Omega}(\bm{L}_{\mathrm{ncvx}}-\bm{M})\|_{\mathrm{F}}\left|\langle\bm{\Delta},\bm{R}-\bm{W}\rangle\right|}_{\eqqcolon\alpha_{2}}.
\end{align*}
We prove in the end of this section that the two terms $\alpha_{1}$
and $\alpha_{2}$ obey
\begin{subequations}
\begin{align}
\alpha_{1} & \leq\lambda^{2}(\sqrt{r}+140\kappa\sigma_{\min}^{-1/2}\|\nabla g(\bm{X},\bm{Y})\|_{\mathrm{F}})^{2}\|\mathcal{P}_{T}(\bm{\Delta})\|_{\mathrm{F}}^{2};\label{eq:Delta-proof-alpha1}\\
\alpha_{2} & \leq560\lambda\kappa\sigma_{\min}^{-1/2}\theta\|\nabla g(\bm{X},\bm{Y})\|_{\mathrm{F}}\|\mathcal{P}_{T}(\bm{\Delta})\|_{\mathrm{F}},\label{eq:Delta-proof-alpha2}
\end{align}
\end{subequations}which yields the upper bound on $\|\mathcal{P}_{\Omega}(\bm{\Delta})\|_{\mathrm{F}}^{2}$
in terms of $\|\mathcal{P}_{T}(\bm{\Delta})\|_{\mathrm{F}}$: 
\begin{align*}
\|\mathcal{P}_{\Omega}(\bm{\Delta})\|_{\mathrm{F}}^{2} & \leq\lambda^{2}(\sqrt{r}+140\kappa\sigma_{\min}^{-1/2}\|\nabla g(\bm{X},\bm{Y})\|_{\mathrm{F}})^{2}\|\mathcal{P}_{T}(\bm{\Delta})\|_{\mathrm{F}}^{2}\\
 & \quad+560\lambda\kappa\sigma_{\min}^{-1/2}\theta\|\nabla g(\bm{X},\bm{Y})\|_{\mathrm{F}}\|\mathcal{P}_{T}(\bm{\Delta})\|_{\mathrm{F}}.
\end{align*}

\paragraph{Step 3: final calculations.}

Using the decomposition $\mathcal{P}_{\Omega}(\bm{\Delta})=\mathcal{P}_{\Omega}\mathcal{P}_{T}(\bm{\Delta})+\mathcal{P}_{\Omega}\mathcal{P}_{T^{\perp}}(\bm{\Delta})$,
we obtain 
\begin{align*}
\|\mathcal{P}_{\Omega}(\bm{\Delta})\|_{\mathrm{F}} & =\|\mathcal{P}_{\Omega}\mathcal{P}_{T}(\bm{\Delta})+\mathcal{P}_{\Omega}\mathcal{P}_{T^{\perp}}(\bm{\Delta})\|_{\mathrm{F}}\\
 & \ge\|\mathcal{P}_{\Omega}\mathcal{P}_{T}(\bm{\Delta})\|_{\mathrm{F}}-\|\mathcal{P}_{\Omega}\mathcal{P}_{T^{\perp}}(\bm{\Delta})\|_{\mathrm{F}}.
\end{align*}
Together with Lemma~\ref{lemma:injectivity} and Equation~\ref{eq:PTperp-Delta-PT-Delta},
we have
\begin{align*}
\|\mathcal{P}_{\Omega}(\bm{\Delta})\|_{\mathrm{F}} & \ge(\sqrt{p}C_{\mathrm{inj}}-140\kappa\sigma_{\min}^{-1/2}\|\nabla g(\bm{X},\bm{Y})\|_{\mathrm{F}})\|\mathcal{P}_{T}(\bm{\Delta})\|_{\mathrm{F}}\\
 & \geq\frac{\sqrt{p}}{2}C_{\mathrm{inj}}\|\mathcal{P}_{T}(\bm{\Delta})\|_{\mathrm{F}}.
\end{align*}
where the last line uses \eqref{eq:gradient-size-bound}. As a result,
we arrive at the sandwhich formula 
\begin{align*}
\frac{1}{4}pC_{\mathrm{inj}}^{2}\|\mathcal{P}_{T}(\bm{\Delta})\|_{\mathrm{F}}^{2} & \le\|\mathcal{P}_{\Omega}(\bm{\Delta})\|_{\mathrm{F}}^{2}\\
 & \le\lambda^{2}(\sqrt{r}+140\kappa\sigma_{\min}^{-1/2}\|\nabla g(\bm{X},\bm{Y})\|_{\mathrm{F}})^{2}\|\mathcal{P}_{T}(\bm{\Delta})\|_{\mathrm{F}}^{2}\\
 & \quad+560\lambda\kappa\sigma_{\min}^{-1/2}\theta\|\nabla g(\bm{X},\bm{Y})\|_{\mathrm{F}}\|\mathcal{P}_{T}(\bm{\Delta})\|_{\mathrm{F}},
\end{align*}
which further implies 
\begin{align*}
 & \left\{ \frac{pC_{\mathrm{inj}}^{2}}{4}-\lambda^{2}(\sqrt{r}+140\kappa\sigma_{\min}^{-1/2}\|\nabla g(\bm{X},\bm{Y})\|_{\mathrm{F}})^{2}\right\} \|\mathcal{P}_{T}(\bm{\Delta})\|_{\mathrm{F}}^{2}\\
 & \quad\leq560\lambda\kappa\sigma_{\min}^{-1/2}\theta\|\nabla g(\bm{X},\bm{Y})\|_{\mathrm{F}}\|\mathcal{P}_{T}(\bm{\Delta})\|_{\mathrm{F}}.
\end{align*}
Reorganize and substitute in \eqref{eq:gradient-size-bound} to see
that for large enough $n$,
\[
\frac{pC_{\mathrm{inj}}^{2}}{4}-\lambda^{2}(\sqrt{r}+140\kappa\sigma_{\min}^{-1/2}\|\nabla g(\bm{X},\bm{Y})\|_{\mathrm{F}})^{2}\geq\frac{pC_{\mathrm{inj}}^{2}}{8}.
\]
Combine the above two relations to reach 
\[
\frac{pC_{\mathrm{inj}}^{2}}{8}\|\mathcal{P}_{T}(\bm{\Delta})\|_{\mathrm{F}}^{2}\le560\lambda\kappa\sigma_{\min}^{-1/2}\theta\|\nabla g(\bm{X},\bm{Y})\|_{\mathrm{F}}\|\mathcal{P}_{T}(\bm{\Delta})\|_{\mathrm{F}},
\]
which together with $C_{\mathrm{inj}}=(32\kappa)^{-1/2}$ and \eqref{eq:theta-bound-Lncvx}
implies 
\[
\|\mathcal{P}_{T}(\bm{\Delta})\|_{\mathrm{F}}\lesssim\frac{\lambda\kappa^{2}}{\sqrt{p\sigma_{\min}}}n\sigma\|\nabla g(\bm{X},\bm{Y})\|_{\mathrm{F}}.
\]
Use \eqref{eq:Delta-PTDelta}, we obtain the bound on $\|\bm{\Delta}\|_{\mathrm{F}}$,

\[
\|\bm{\Delta}\|_{\mathrm{F}}\le2\|\mathcal{P}_{T}(\bm{\Delta})\|_{\mathrm{F}}\lesssim\frac{\lambda\kappa^{2}}{\sqrt{p\sigma_{\min}}}n\sigma\|\nabla g(\bm{X},\bm{Y})\|_{\mathrm{F}}.
\]

\paragraph{Proof of the bound~\eqref{eq:Delta-proof-alpha1}.}

For $\alpha_{1}$ we consider the cases when $\|\mathcal{P}_{\Omega}(\bm{L}_{\mathrm{cvx}}-\bm{M})\|_{\mathrm{F}}-\|\mathcal{P}_{\Omega}(\bm{L}_{\mathrm{ncvx}}-\bm{M})\|_{\mathrm{F}}$
is positive and non-positive separately. 

\paragraph{Case of $\|\mathcal{P}_{\Omega}(\bm{L}_{\mathrm{cvx}}-\bm{M})\|_{\mathrm{F}}-\|\mathcal{P}_{\Omega}(\bm{L}_{\mathrm{ncvx}}-\bm{M})\|_{\mathrm{F}}\le0$.}

By convexity of $\|\cdot\|_{\mathrm{F}}$,
\[
0\ge\|\mathcal{P}_{\Omega}(\bm{L}_{\mathrm{cvx}}-\bm{M})\|_{\mathrm{F}}-\|\mathcal{P}_{\Omega}(\bm{L}_{\mathrm{ncvx}}-\bm{M})\|_{\mathrm{F}}>\left\langle \frac{1}{\theta}\mathcal{P}(\bm{L}_{\mathrm{ncvx}}-\bm{M}),\bm{\Delta}\right\rangle .
\]
Using the representation in Lemma~\ref{lemma:residual-size-in-T},
the last term can be writen as $\lambda\left\langle \bm{U}\bm{V}^{\top}+\bm{R},\bm{\Delta}\right\rangle $.
Splitting the parts into $T$ and $T^{\perp}$, we have
\begin{align*}
 & \left(\|\mathcal{P}_{\Omega}(\bm{L}_{\mathrm{cvx}}-\bm{M})\|_{\mathrm{F}}-\|\mathcal{P}_{\Omega}(\bm{L}_{\mathrm{ncvx}}-\bm{M})\|_{\mathrm{F}}\right)^{2}\\
 & \quad\le\lambda^{2}\left\langle \bm{U}\bm{V}^{\top}+\bm{R},\bm{\Delta}\right\rangle ^{2}\\
 & \quad\le\lambda^{2}\left(\|\bm{U}\bm{V}^{\top}\|_{\mathrm{F}}\|\mathcal{P}_{T}(\bm{\Delta})\|_{\mathrm{F}}+\|\mathcal{P}_{T}(\bm{R})\|_{\mathrm{F}}\|\mathcal{P}_{T}(\bm{\Delta})\|_{\mathrm{F}}+\|\mathcal{P}_{T^{\perp}}(\bm{R})\|\|\mathcal{P}_{T^{\perp}}(\bm{\Delta})\|_{*}\right)^{2}\text{.}
\end{align*}
Together with Equation \eqref{eq:PTperp-Delta-PT-Delta} and Lemma~\ref{lemma:residual-size-in-T},
we arrive at
\[
\left(\|\mathcal{P}_{\Omega}(\bm{L}_{\mathrm{cvx}}-\bm{M})\|_{\mathrm{F}}-\|\mathcal{P}_{\Omega}(\bm{L}_{\mathrm{ncvx}}-\bm{M})\|_{\mathrm{F}}\right)^{2}\le\lambda^{2}(\sqrt{r}+140\kappa\sigma_{\min}^{-1/2}\|\nabla g(\bm{X},\bm{Y})\|_{\mathrm{F}})^{2}\|\mathcal{P}_{T}(\bm{\Delta})\|_{\mathrm{F}}^{2}.
\]

\paragraph{Case of $\|\mathcal{P}_{\Omega}(\bm{L}_{\mathrm{cvx}}-\bm{M})\|_{\mathrm{F}}-\|\mathcal{P}_{\Omega}(\bm{L}_{\mathrm{ncvx}}-\bm{M})\|_{\mathrm{F}}>0$.}

By optimality of $\bm{L}_{\mathrm{cvx}}$ and convexity of $\|\cdot\|_{\star}$,
\begin{align*}
0 & <\|\mathcal{P}_{\Omega}(\bm{L}_{\mathrm{cvx}}-\bm{M})\|_{\mathrm{F}}-\|\mathcal{P}_{\Omega}(\bm{L}_{\mathrm{ncvx}}-\bm{M})\|_{\mathrm{F}}\le-\lambda\left(\|\bm{L}_{\mathrm{cvx}}\|_{*}-\|\bm{L}_{\mathrm{ncvx}}\|_{*}\right)\le-\lambda\left\langle \bm{U}\bm{V}^{\top},\bm{\Delta}\right\rangle .
\end{align*}
Then similar to the case of $\|\mathcal{P}_{\Omega}(\bm{L}_{\mathrm{cvx}}-\bm{M})\|_{\mathrm{F}}-\|\mathcal{P}_{\Omega}(\bm{L}_{\mathrm{ncvx}}-\bm{M})\|_{\mathrm{F}}\le0$,
\begin{align*}
 & \left(\|\mathcal{P}_{\Omega}(\bm{L}_{\mathrm{cvx}}-\bm{M})\|_{\mathrm{F}}-\|\mathcal{P}_{\Omega}(\bm{L}_{\mathrm{ncvx}}-\bm{M})\|_{\mathrm{F}}\right)^{2}\le\lambda^{2}r\|\mathcal{P}_{T}(\bm{\Delta})\|_{\mathrm{F}}^{2}.
\end{align*}
Combining the two cases yields \eqref{eq:Delta-proof-alpha1}.

\paragraph{Proof of the bound~\eqref{eq:Delta-proof-alpha2}.}

For $\alpha_{2}$, we can split the parts into $T$ and $T^{\perp}$
similar to the proof for \eqref{eq:Delta-proof-alpha1}. Using Equation~\eqref{eq:PTperp-Delta-PT-Delta}
and Lemma~\ref{lemma:residual-size-in-T}, we have
\begin{align}
2\theta\cdot\lambda\langle\bm{\Delta},\bm{R}-\bm{W}\rangle & \le2\lambda\theta\left(\left|\langle\bm{\Delta},\bm{R}\rangle\right|+\left|\langle\bm{\Delta},\bm{W}\rangle\right|\right)\label{eq:Delta-Omega-to-T-ub-inner-product-part}\\
 & \le2\lambda\theta\left[\|\mathcal{P}_{T}(\bm{R})\|_{\mathrm{F}}\|\mathcal{P}_{T}(\bm{\Delta})\|_{\mathrm{F}}+\left(\|\mathcal{P}_{T^{\perp}}(\bm{R})\|+\|\mathcal{P}_{T^{\perp}}(\boldsymbol{W})\|\right)\|\mathcal{P}_{T^{\perp}}(\bm{\Delta})\|_{*}\right]\nonumber \\
 & \le560\lambda\kappa\sigma_{\min}^{-1/2}\theta\|\nabla g(\bm{X},\bm{Y})\|_{\mathrm{F}}\|\mathcal{P}_{T}(\bm{\Delta})\|_{\mathrm{F}}.\nonumber 
\end{align}

\subsection{Proof of Lemma \ref{lemma:residual-size-in-T} \label{subsec:Proof-Gradient-Decomp}}

The proof relies on the following representation of the low-rank factors
$\bm{X},\bm{Y}$ of the nonconvex solution $\bm{L}_{\mathrm{ncvx}}$. 

\begin{lemma}\label{lemma:XY-decomposition-Q}

Under the assumptions and notations of Lemma \ref{lemma:ncvx-to-cvx-full},
there exists an invertible matrix $\bm{Q}\in\mathbb{R}^{r\times r}$
such that $\bm{X}=\bm{U}\bm{\Sigma}^{1/2}\bm{Q},\bm{Y}=\bm{V}\bm{\Sigma}^{1/2}\bm{Q}^{-\top}$,
$\|\bm{Q}\|\le2$ and 
\begin{equation}
\left\Vert \bm{\Sigma}^{1/2}\bm{Q}\bm{Q}^{\top}\bm{\Sigma}^{-1/2}-\bm{I}_{r}\right\Vert \le\frac{32\kappa}{\sqrt{\sigma_{\min}}}\|\nabla g(\bm{X},\bm{Y})\|_{\mathrm{F}}\le1/3.\label{eq:QQT-I}
\end{equation}
where $\bm{U}_{\bm{Q}}\bm{\Sigma}_{\bm{Q}}\bm{V}_{\bm{Q}}$ is the
SVD of $\bm{Q}$.

\end{lemma}

\noindent See Section~\ref{subsec:Proof-XY-decomp} for the proof.

\medskip

Denote the partial gradients of $g(\bm{X},\bm{Y})$ as $\bm{B}_{1},\bm{B}_{2}$,
i.e.,
\begin{align}
\bm{B}_{1} & \coloneqq\nabla_{\bm{X}}g(\bm{X},\bm{Y})=\frac{1}{\theta}\mathcal{P}_{\Omega}(\bm{X}\bm{Y}^{\top}-\bm{M})\bm{Y}+\lambda\bm{X};\label{eq:f-gradients-B1}\\
\bm{B}_{2} & \coloneqq\nabla_{\bm{Y}}g(\bm{X},\bm{Y})=\frac{1}{\theta}\mathcal{P}_{\Omega}(\bm{X}\bm{Y}^{\top}-\bm{M})^{\top}\bm{X}+\lambda\bm{Y},\label{eq:f-gradient-B2}
\end{align}
where we recall $\theta=\|\mathcal{P}_{\Omega}(\bm{X}\bm{Y}^{\top}-\bm{M})\|_{\mathrm{F}}$.
By definition, we know that $\max\left\{ \|\bm{B}_{1}\|_{\mathrm{F}},\|\bm{B}_{2}\|_{\mathrm{F}}\right\} \le\|\nabla g(\bm{X},\bm{Y})\|_{\mathrm{F}}$. 

Let $\bm{R}$ be the matrix that is defined by equation~\eqref{eq:f-gradient-residual}.
We now control its component in $T$ and $T^{\perp}$ separately.

\paragraph{Part 1: Bounding $\|\mathcal{P}_{T}(\bm{R})\|_{\mathrm{F}}$. }

By the definition of the projection operator $\mathcal{P}_{T}$, we
have
\begin{align*}
\|\mathcal{P}_{T}(\bm{R})\|_{\mathrm{F}} & =\|\bm{U}\bm{U}^{\top}\bm{R}(\bm{I}-\bm{V}\bm{V}^{\top})+\boldsymbol{RVV}^{\top}\|_{\mathrm{F}}\\
 & \le\|\bm{U}\bm{U}^{\top}\bm{R}(\bm{I}-\bm{V}\bm{V}^{\top})\|_{\mathrm{F}}+\|\boldsymbol{RVV}^{\top}\|_{\mathrm{F}}\\
 & \le\|\bm{U}^{\top}\bm{R}\|_{\mathrm{F}}+\|\bm{RV}\|_{\mathrm{F}}.
\end{align*}
For the term $\bm{R}\bm{V}$, we use the definitions of $\bm{B}_{1}$
and $\bm{R}$ to see that 
\[
\lambda\bm{U}\bm{V}^{\top}\bm{Y}+\lambda\bm{R}\bm{Y}=\lambda\bm{X}-\bm{B}_{1},
\]
which together with the representations in Lemma~\ref{lemma:XY-decomposition-Q}
implies
\[
\bm{R}\bm{V}=\bm{U}\bm{\Sigma}^{1/2}(\bm{Q}\bm{Q}^{\top}-\bm{I}_{r})\bm{\Sigma}^{-1/2}-\bm{B}_{1}\bm{Q}^{\top}\bm{\Sigma}^{-1/2}.
\]
In view of the relation~\eqref{eq:QQT-I}, we have
\begin{align*}
\|\bm{R}\bm{V}\|_{\mathrm{F}} & \le\|\bm{\Sigma}^{1/2}(\bm{Q}\bm{Q}^{\top}-\bm{I}_{r})\bm{\Sigma}^{-1/2}\|_{\mathrm{F}}+\|\bm{\Sigma}^{-1/2}\|\|\bm{Q}\|\|\bm{B}_{1}\|_{\mathrm{F}}\\
 & \le\frac{32\kappa}{\sqrt{\sigma_{\min}}}\|\nabla g(\bm{X},\bm{Y})\|_{\mathrm{F}}+2\sqrt{\frac{2}{\sigma_{\min}}}\|\nabla g(\bm{X},\bm{Y})\|_{\mathrm{F}}\\
 & \le\frac{35\kappa}{\sqrt{\sigma_{\min}}}\|\nabla g(\bm{X},\bm{Y})\|_{\mathrm{F}},
\end{align*}
where we have used the fact that $\|\bm{\Sigma}^{-1}\|\leq\sigma_{\min}/2$.
Similarly we can establish that \[\|\bm{U}^{\top}\bm{R}\|_{\mathrm{F}}\le\frac{35\kappa}{\sqrt{\sigma_{\min}}}\|\nabla g(\bm{X},\bm{Y})\|_{\mathrm{F}}.\]
Combine the two inequalities to arrive at 
\[
\|\mathcal{P}_{T}(\bm{R})\|_{\mathrm{F}}\le\frac{70\kappa}{\sqrt{\sigma_{\min}}}\|\nabla g(\bm{X},\bm{Y})\|_{\mathrm{F}}.
\]

\paragraph{Part 2: Bounding $\|\mathcal{P}_{T^{\perp}}(\bm{R})\|$.}

For any matrix $\bm{A}$, define $\mathcal{P}_{\Omega}^{\mathrm{debias}}(\bm{A})\coloneqq\mathcal{P}_{\Omega}(\bm{A})-p\bm{A}$.
We can rewrite the identities~\eqref{eq:f-gradients-B1} and \eqref{eq:f-gradient-B2}
as 
\begin{align*}
\frac{1}{\theta}\left[p\bm{L}^{\star}+\mathcal{P}_{\Omega}(\bm{E})-\mathcal{P}_{\Omega}^{\mathrm{debias}}(\bm{X}\bm{Y}^{\top}-\bm{L}^{\star})\right]\bm{Y} & =\frac{p}{\theta}\bm{X}\bm{Y}^{\top}\bm{Y}+\lambda\bm{X}-\bm{B}_{1};\\
\frac{1}{\theta}\left[p\bm{L}^{\star}+\mathcal{P}_{\Omega}(\bm{E})-\mathcal{P}_{\Omega}^{\mathrm{debias}}(\bm{X}\bm{Y}^{\top}-\bm{L}^{\star})\right]^{\top}\bm{X} & =\frac{p}{\theta}\bm{Y}\bm{X}^{\top}\bm{X}+\lambda\bm{Y}-\bm{B}_{2}.
\end{align*}
Again, using the representations in Lemma~\ref{lemma:XY-decomposition-Q},
we have the following two identities
\begin{subequations}
\begin{align}
\frac{1}{\theta}\left[p\bm{L}^{\star}+\mathcal{P}_{\Omega}(\bm{E})-\mathcal{P}_{\Omega}^{\mathrm{debias}}(\bm{X}\bm{Y}^{\top}-\bm{L}^{\star})\right]\bm{V} & =\frac{1}{\theta}p\bm{U}\bm{\Sigma}+\lambda\bm{U}\bm{\Sigma}^{1/2}\bm{Q}\bm{Q}^{\top}\bm{\Sigma}^{-1/2}-\bm{B}_{1}\bm{Q}^{\top}\bm{\Sigma}^{-1/2};\label{eq:bounding-tildeR-V}\\
\frac{1}{\theta}\left[p\bm{L}^{\star}+\mathcal{P}_{\Omega}(\bm{E})-\mathcal{P}_{\Omega}^{\mathrm{debias}}(\bm{X}\bm{Y}^{\top}-\bm{L}^{\star})\right]^{\top}\bm{U} & =\frac{1}{\theta}p\bm{V}\bm{\Sigma}+\lambda\bm{V}\bm{\Sigma}^{1/2}\bm{Q}^{-\top}\bm{Q}^{-1}\bm{\Sigma}^{-1/2}-\bm{B}_{2}\bm{Q}^{-1}\bm{\Sigma}^{-1/2}.\label{eq:bounding-tildeR-U}
\end{align}
\end{subequations}

These two equations motivate us to define a matrix $\tilde{\bm{R}}$
using
\begin{equation}
\frac{1}{\theta}\left[p\bm{L}^{\star}+\mathcal{P}_{\Omega}(\bm{E})-\mathcal{P}_{\Omega}^{\mathrm{debias}}(\bm{X}\bm{Y}^{\top}-\bm{L}^{\star})\right]=\frac{1}{\theta}p\bm{U}\bm{\Sigma}\bm{V}^{\top}+\lambda\bm{U}\bm{\Sigma}^{1/2}\bm{Q}\bm{Q}^{\top}\bm{\Sigma}^{-1/2}\bm{V}^{\top}+\lambda\tilde{\bm{R}},\label{eq:def-R-tilde}
\end{equation}
where $\tilde{\bm{R}}$ obeys $\mathcal{P}_{T^{\perp}}(\bm{R})=\mathcal{P}_{T^{\perp}}(\tilde{\bm{R}})$.
To see this, we use the definition of $\bm{R}$ to write 
\begin{align}
\mathcal{P}_{T^{\perp}}(\bm{R}) & =-\frac{1}{\lambda}\mathcal{P}_{T^{\perp}}\left(\theta^{-1}\mathcal{P}_{\Omega}(\bm{X}\bm{Y}^{\top}-\bm{M})\right)=-\frac{1}{\lambda\theta}\mathcal{P}_{T^{\perp}}\left[\mathcal{P}_{\Omega}(\bm{X}\bm{Y}^{\top}-\bm{L}^{\star})-\mathcal{P}_{\Omega}(\bm{E})\right].\label{eq:R-tildeR}
\end{align}
Since $\mathcal{P}_{T^{\perp}}(\bm{X}\bm{Y}^{\top})=0$, by definition
of $\tilde{\bm{R}}$, we obtain
\begin{align*}
\mathcal{P}_{T^{\perp}}(\bm{R}) & =\frac{1}{\lambda\theta}\mathcal{P}_{T^{\perp}}\left[p(\bm{L}^{\star}-\bm{X}\bm{Y}^{\top})+\mathcal{P}_{\Omega}(\bm{E})-\mathcal{P}_{\Omega}(\bm{X}\bm{Y}^{\top}-\bm{L}^{\star})\right]=\mathcal{P}_{T^{\perp}}(\tilde{\bm{R}}).
\end{align*}
Therefore from now on, we concentrate on bounding $\|\mathcal{P}_{T^{\perp}}(\tilde{\bm{R}})\|$. 

To this end, we rewrite equation~\eqref{eq:def-R-tilde} as 
\[
\frac{1}{\theta}\left[p\bm{L}^{\star}+\mathcal{P}_{\Omega}(\bm{E})-\mathcal{P}_{\Omega}^{\mathrm{debias}}(\bm{X}\bm{Y}^{\top}-\bm{L}^{\star})\right]-\lambda\mathcal{P}_{T}(\tilde{\bm{R}})=\frac{1}{\theta}p\bm{U}\bm{\Sigma}\bm{V}^{\top}+\lambda\bm{U}\bm{\Sigma}^{1/2}\bm{Q}\bm{Q}^{\top}\bm{\Sigma}^{-1/2}\bm{V}^{\top}+\lambda\mathcal{P}_{T^{\perp}}(\tilde{\bm{R}}).
\]
Suppose that 
\[
\|\mathcal{P}_{T}(\tilde{\bm{R}})\|\leq\frac{\lambda}{4}\theta,
\]
which together with Lemma~\ref{lemma:E-bound} and Lemma~\ref{lemma:error-size}
implies that 
\[
\frac{1}{\theta}\left\Vert \mathcal{P}_{\Omega}(\bm{E})-\mathcal{P}_{\Omega}^{\mathrm{debias}}(\bm{X}\bm{Y}^{\top}-\bm{L}^{\star})-\lambda\mathcal{P}_{T}(\tilde{\bm{R}})\right\Vert _{\mathrm{}}\le\lambda/8+\lambda/8+\lambda/4=\lambda/2.
\]
By Weyl's inequality and the fact that $\bm{L}^{\star}$ is of rank
$r$, for each $i=r+1,\ldots,n$, one has 
\begin{align}
 & \sigma_{i}\left(\frac{1}{\theta}p\bm{U}\bm{\Sigma}\bm{V}^{\top}+\lambda\bm{U}\bm{\Sigma}^{1/2}\bm{Q}\bm{Q}^{\top}\bm{\Sigma}^{-1/2}\bm{V}^{\top}+\lambda\mathcal{P}_{T^{\perp}}(\tilde{\bm{R}})\right)\label{eq:singular-value-ub}\\
 & \quad\le\frac{1}{\theta}\left\Vert \mathcal{P}_{\Omega}(\bm{E})-\mathcal{P}_{\Omega}^{\mathrm{debias}}(\bm{X}\bm{Y}^{\top}-\bm{L}^{\star})-\lambda\mathcal{P}_{T}(\tilde{\bm{R}})\right\Vert \\
 & \quad\le\lambda/2.\nonumber 
\end{align}
At the same time, for each $i=1,\ldots,r$, we have
\begin{align}
 & \sigma_{i}\left(\frac{1}{\theta}p\bm{U}\bm{\Sigma}\bm{V}^{\top}+\lambda\bm{U}\bm{\Sigma}^{1/2}\bm{Q}\bm{Q}^{\top}\bm{\Sigma}^{-1/2}\bm{V}^{\top}\right)\label{eq:singular-value-lb}\\
 & \quad\ge\sigma_{r}\left[\bm{U}\left(\frac{1}{\theta}p\bm{\Sigma}+\lambda\bm{I}_{r}+\lambda(\bm{\Sigma}^{1/2}\bm{Q}\bm{Q}^{\top}\bm{\Sigma}^{-1/2}-\bm{I}_{r})\right)\bm{V}^{\top}\right]\nonumber \\
 & \quad\ge\sigma_{r}\left(\frac{1}{\theta}p\bm{\Sigma}+\lambda\bm{I}_{r}\right)-\lambda\left\Vert \bm{\Sigma}^{1/2}\bm{Q}\bm{Q}^{\top}\bm{\Sigma}^{-1/2}-\bm{I}_{r}\right\Vert \nonumber \\
 & \quad\ge\lambda-\lambda/3>\lambda/2,\nonumber 
\end{align}
where the last line uses the claim~\eqref{eq:QQT-I}. As a result,
the singular values of $\lambda\mathcal{P}_{T^{\perp}}(\tilde{\bm{R}})$
must fall below $\lambda/2$, i.e., 
\[
\|\mathcal{P}_{T^{\perp}}(\bm{R})\|=\|\mathcal{P}_{T^{\perp}}(\tilde{\bm{R}})\|<1/2.
\]
We are left with controlling $\|\mathcal{P}_{T}(\tilde{\bm{R}})\|$.
Similar to bounding $\|\mathcal{P}_{T}(\bm{R})\|$, using \eqref{eq:bounding-tildeR-V}
and \eqref{eq:bounding-tildeR-U} we have 
\begin{align*}
\|\tilde{\bm{R}}\bm{V}\|_{\mathrm{F}} & =\frac{1}{\lambda}\|\bm{B}_{1}\bm{Q}^{\top}\bm{\Sigma}^{-1/2}\bm{V}\|_{\mathrm{F}}\\
 & \le\frac{1}{\lambda}\|\bm{Q}\|\|\bm{\Sigma}^{-1/2}\|\|\bm{B}_{1}\|_{\mathrm{F}}\\
 & \le\frac{2}{\lambda\sqrt{\sigma_{\min}/2}}\|\nabla g(\bm{X},\bm{Y})\|_{\mathrm{F}}
\end{align*}
and 
\begin{align*}
\|\tilde{\bm{R}}^{\top}\bm{U}\|_{\mathrm{F}} & =\|\bm{V}(\bm{\Sigma}^{-1/2}\bm{Q}\bm{Q}^{\top}\bm{\Sigma}^{1/2}-\bm{\Sigma}^{1/2}\bm{Q}^{-\top}\bm{Q}^{-1}\bm{\Sigma}^{-1/2})-\frac{1}{\lambda}\bm{B}_{2}\bm{Q}^{\top}\bm{\Sigma}^{-1/2}\bm{U}\|_{\mathrm{F}}\\
 & \le\|\bm{\Sigma}^{-1/2}(\bm{Q}\bm{Q}^{\top}-\bm{I}_{r})\bm{\Sigma}^{1/2}\|_{\mathrm{F}}+\|\bm{\Sigma}^{1/2}(\bm{Q}^{-\top}\bm{Q}^{-1}-\bm{I}_{r})\bm{\Sigma}^{-1/2}\|_{\mathrm{F}}+\frac{1}{\lambda}\|\bm{B}_{2}\bm{Q}^{\top}\bm{\Sigma}^{-1/2}\bm{U}\|_{\mathrm{F}}\\
 & \le\frac{64\kappa}{\sqrt{\sigma_{\min}}}\|\nabla g(\bm{X},\bm{Y})\|_{\mathrm{F}}+\frac{2}{\lambda\sqrt{\sigma_{\min}/2}}\|\nabla g(\bm{X},\bm{Y})\|_{\mathrm{F}}.
\end{align*}
Combining the two bounds we have 
\begin{align*}
\|\mathcal{P}_{T}(\tilde{\bm{R}})\| & \le\|\mathcal{P}_{T}(\tilde{\bm{R}})\|_{\mathrm{F}}\le\|\tilde{\bm{R}}^{\top}\bm{U}\|_{\mathrm{F}}+\|\tilde{\bm{R}}\bm{V}\|_{\mathrm{F}}\\
 & \le\frac{64\kappa+8/\lambda}{\sqrt{\sigma_{\min}}}\|\nabla g(\bm{X},\bm{Y})\|_{\mathrm{F}}\\
 & \le\frac{\lambda\theta}{4},
\end{align*}
where the last line comes from equation~\eqref{eq:gradient-size-bound}
and Lemma~\ref{lemma:error-size}.

\subsection{Proof of Lemma~\ref{lemma:XY-decomposition-Q}\label{subsec:Proof-XY-decomp}}

Reuse the definitions of $\bm{B}_{1},\bm{B}_{2}$ in equations~\eqref{eq:f-gradients-B1}
and \eqref{eq:f-gradient-B2}. We can then write
\begin{align*}
\bm{X}^{\top}\bm{X}-\bm{Y}^{\top}\bm{Y} & =\frac{1}{\lambda}\left[\bm{X}^{\top}\left(\bm{B}_{1}-\frac{1}{\theta}\mathcal{P}(\bm{X}\bm{Y}^{\top}-\bm{M})\bm{Y}\right)-\left(\bm{B}_{2}-\frac{1}{\theta}\mathcal{P}(\bm{X}\bm{Y}^{\top}-\bm{M})^{\top}\bm{X}\right){}^{\top}\bm{Y}\right]\\
 & =\frac{1}{\lambda}\left(\bm{X}^{\top}\bm{B}_{1}-\bm{B}_{2}^{\top}\bm{Y}\right),
\end{align*}
which further implies
\begin{align*}
\|\bm{X}^{\top}\bm{X}-\bm{Y}^{\top}\bm{Y}\|_{\mathrm{F}} & =\frac{1}{\lambda}\left\Vert \bm{X}^{\top}\bm{B}_{1}-\bm{B}_{2}^{\top}\bm{Y}\right\Vert _{\mathrm{F}} \le\frac{1}{\lambda}\left(\|\bm{X}\|\|\bm{B}_{1}\|_{\mathrm{F}}+\|\bm{B}_{2}\|_{\mathrm{F}}\|\bm{Y}\|\right)\\
 & \le\frac{2\sqrt{2\sigma_{\max}}}{\lambda}\|\nabla g(\bm{X},\bm{Y})\|_{\mathrm{F}}.
\end{align*}
Here, the last inequality uses the fact that $\max\left\{ \|\bm{B}_{1}\|_{\mathrm{F}},\|\bm{B}_{2}\|_{\mathrm{F}}\right\} \le\|\nabla g(\bm{X},\bm{Y})\|_{\mathrm{F}}$,
and that \[\max\{\|\bm{X}\|,\|\bm{Y}\|\}\leq\sqrt{2\sigma_{\max}}.\] 

In addition, since $\min\{\sigma_{\min}(\bm{X}),\sigma_{\min}(\bm{Y})\}\geq\sqrt{\sigma_{\min}/2}$,
we have $\sigma_{\min}(\bm{X}\bm{Y}^{\top})\geq\sigma_{\min}/2$,
which together with Lemma 20 in the paper \cite{chen2020noisy} implies
the existence of an invertible $\bm{Q}\in\mathbb{R}^{r\times r}$
such that $\bm{X}=\bm{U}\bm{\Sigma}^{1/2}\bm{Q},\bm{Y}=\bm{V}\bm{\Sigma}^{1/2}\bm{Q}^{-\top}$,
and 
\begin{align*}
\|\bm{\Sigma}_{\bm{Q}}-\bm{\Sigma}_{\bm{Q}}^{-1}\|_{\mathrm{F}} & \le\frac{2}{\sigma_{\min}}\|\bm{X}^{\top}\bm{X}-\bm{Y}^{\top}\bm{Y}\|_{\mathrm{F}}\\
 & \le\frac{4\sqrt{2\sigma_{\max}}}{\lambda\sigma_{\min}}\|\nabla g(\bm{X},\bm{Y})\|_{\mathrm{F}}=\frac{4\sqrt{2\kappa}}{\lambda\sqrt{\sigma_{\min}}}\|\nabla g(\bm{X},\bm{Y})\|_{\mathrm{F}}.
\end{align*}
In view of the assumed upper bound $\|\nabla g(\bm{X},\bm{Y})\|_{\mathrm{F}}\leq C_\mathrm{grad}\frac{1}{n^8}\sqrt{\frac{\sigma_{\max}}{p}}$ and $n^2 p\gg\kappa$,
one has 
\[
\sigma_{\max}(\bm{\Sigma}_{\bm{Q}})-\sigma_{\max}^{-1}(\bm{\Sigma}_{\bm{Q}})\le\|\bm{\Sigma}_{\bm{Q}}-\bm{\Sigma}_{\bm{Q}}^{-1}\|_{\mathrm{F}}\le C_\mathrm{grad}\frac{4\sqrt{2}\cdot\kappa}{\lambda n^8\sqrt{p}}\le1,
\]
and hence $\|\bm{Q}\|=\|\bm{\Sigma}_{\bm{Q}}\|=\sigma_{\max}(\bm{\Sigma}_{\bm{Q}})\le2$.
As a result, we have 
\begin{align*}
\left\Vert \bm{\Sigma}^{1/2}\bm{Q}\bm{Q}^{\top}\bm{\Sigma}^{-1/2}-\bm{I}_{r}\right\Vert  & =\left\Vert \bm{\Sigma}^{1/2}(\bm{U}_{\bm{Q}}\bm{\Sigma}_{\bm{Q}}\bm{\Sigma}_{\bm{Q}}\bm{U}_{\bm{Q}}^{\top}-\bm{U}_{\bm{Q}}\bm{\Sigma}_{\bm{Q}}\bm{\Sigma}_{\bm{Q}}^{-1}\bm{U}_{\bm{Q}}^{\top})\bm{\Sigma}^{-1/2}\right\Vert \\
 & \le\|\bm{\Sigma}^{1/2}\|\|\bm{\Sigma}^{-1/2}\|\left\Vert \bm{U}_{\bm{Q}}\right\Vert \|\bm{U}_{\bm{Q}}^{\top}\|\|\boldsymbol{\Sigma}_{\boldsymbol{Q}}\|\|\bm{\Sigma}_{\bm{Q}}-\bm{\Sigma}_{\bm{Q}}^{-1}\|_{\mathrm{F}}\\
 & \le\frac{32\kappa}{\sqrt{\sigma_{\min}}}\|\nabla g(\bm{X},\bm{Y})\|_{\mathrm{F}}\le1/3,
\end{align*}
where the last inequality again uses the assumed bound~\eqref{eq:gradient-size-bound}.
\end{document}